\documentclass[11pt]{article}
\usepackage[T1]{fontenc}
\usepackage[utf8]{inputenc}
\usepackage{lmodern}
\usepackage[a4paper,margin=1in]{geometry}
\usepackage{amsmath,amssymb,amsthm}
\usepackage{xcolor}
\usepackage{microtype}
\usepackage{enumitem}
\usepackage[hidelinks]{hyperref}
\hypersetup{pdftitle={Periodic Tilings of Cardinality Twice a Prime or Nine in Arbitrary Dimension},
  pdfauthor={Hu Tan and Ying Zhang}}
\newtheorem{theorem}{Theorem}[section]
\newtheorem{proposition}[theorem]{Proposition}
\newtheorem{lemma}[theorem]{Lemma}
\newtheorem{corollary}[theorem]{Corollary}
\theoremstyle{definition}

\theoremstyle{remark}
\newtheorem{remark}[theorem]{Remark}

\title{Periodic Tilings of Cardinality Twice a Prime or Nine \\ in Arbitrary Dimension}

\date{}

\begin{document}
\maketitle
\vspace{-1.5em}
\author{
\begin{tabular}[t]{c} 
Hu Tan\textsuperscript{1}\\[3pt] 
\small\textsuperscript{1}Academy of Mathematics and Systems Science\\ 
\small Chinese Academy of Sciences\\ 
\small Beijing 100190, China\\ 
\small\texttt{tanhu2020@amss.ac.cn} 
\end{tabular} 
\and 
\begin{tabular}[t]{c} 
Ying Zhang\textsuperscript{2,}\\[3pt] 
\small\textsuperscript{2}School of Mathematical Sciences\\ 
\small Soochow University\\ 
\small Suzhou 215006, China\\ 
\small \texttt{yzhang@suda.edu.cn} 
\end{tabular}
}
\vspace{5em}
\begin{abstract}

We prove that every finite translational tile of $\mathbb Z^d$ of
cardinality $2q$, where $q$ is an odd prime, or of cardinality nine
admits a fully periodic tiling complement. The result holds in every
dimension, regardless of the rank of the subgroup generated by the
differences of points of the tile. The proofs combine coprime dilation,
vanishing sums of roots of unity, and periodic two-colorings. After
translating the tile, both arguments restrict to its intrinsic lattice,
and every periodic complement constructed there extends to the ambient
lattice. For cardinality $2q$, spectral filtering in the intrinsic
lattice yields either a lattice complement or a periodic twofold
covering compatible with a tiling complement. This compatibility allows
the covering to be split by a periodic proper two-coloring. For nine
points, intrinsic ranks $1$ and $2$ are treated separately; in higher
intrinsic rank, the algebraic reduction in the intrinsic lattice yields
either a lattice complement or spectral support on finitely many
rational affine circles. These results also give a decision algorithm
for tileability in both cardinality families. The circle-supported
alternative is resolved by a conditional-density dichotomy and a
periodic replacement argument for the exceptional half-density components.

\end{abstract}


\section{Introduction}\label{sec:intro}

Translational tiling lies at the intersection of discrete geometry,
additive combinatorics, harmonic analysis, and the factorization
theory of abelian groups. Let $d\geq1$ be an integer. In the discrete
setting, a finite set $F\subseteq\mathbb Z^d$ \emph{tiles by translations}
if there is a set $A\subseteq\mathbb Z^d$ such that every lattice point
has a unique representation $x=f+a$, with $f\in F$ and $a\in A$.
We write $F\oplus A=\mathbb Z^d$ and call $A$ a tiling complement
(or co-tile). The subgroup
$\langle F-F\rangle\leq\mathbb Z^d$ generated by the differences of
points of $F$ is the \emph{intrinsic lattice} of $F$, and its rank is
the \emph{intrinsic rank}. This viewpoint is rooted in the classical
theory of factorizations of finite and cyclic groups and its connections
with geometric tiling problems; see
\cite{CM,deBruijn,GrunbaumShephard,Stein}.

A tiling complement $A\subseteq\mathbb Z^d$ is \emph{fully periodic}
if $A+\Gamma=A$ for some finite-index subgroup
$\Gamma\leq\mathbb Z^d$. A fully periodic complement is determined by
a finite pattern in a finite quotient of $\mathbb Z^d$. The periodic tiling
problem asks whether every finite tile admitting a tiling complement
also admits a fully periodic one. This question has its roots in the
classical theory of translational tilings; see
\cite{GrunbaumShephard,LagariasWang} and the historical account in
\cite{GreenfeldSurvey}.

Periodic complements connect infinite exact-covering problems to
factorizations of finite abelian groups. In particular, they provide
finite certificates for tileability. This connection underlies
decision procedures for classes of tiles for which periodic
tileability holds; see
\cite{Bhattacharya,GTCounterexample,GTStructure,Szegedy}.

\subsection{Background and main result}

The one-dimensional theory is completely periodic. Newman proved
that every tiling of $\mathbb Z$ by a finite set is periodic
\cite{Newman}, while later work gave finer information on the
arithmetic structure and possible periods of integer tilings
\cite{CM,LZ}. In dimension two, Bhattacharya proved that every finite
tile admitting a tiling also admits a fully periodic complement
\cite{Bhattacharya}. Greenfeld and Tao subsequently obtained
quantitative periodicization results and a stronger structural
description: every planar tiling complement is a finite disjoint
union of sets, each invariant under a nonzero translation
\cite{GTStructure}.

The picture changes in high dimension. Greenfeld and Tao constructed,
in sufficiently large dimension, a finite translational tile that
admits tilings but no fully periodic complement
\cite{GTCounterexample}. Thus periodic tileability cannot hold in
complete generality, and the problem becomes one of identifying
natural hypotheses that force periodicity. Low dimension remains a
major source of rigidity; in particular, the general periodic tiling
problem for $\mathbb Z^3$ remains open; see
\cite[Question~5.1]{GreenfeldSurvey}.

A second source of rigidity is the cardinality of the tile. Szegedy
proved periodic tileability in every dimension for tiles of prime
cardinality and for four-point tiles \cite{Szegedy}. For
prime-cardinality tiles containing zero and generating their ambient
lattice, a lattice complement exists; an explicit algebraic proof
is given by Horak and Kim \cite[Theorem~19]{HorakKim}, who also
note the earlier statement in \cite{Szegedy}. These results
show that arithmetic restrictions on $|F|$ can force periodicity
independently of the ambient dimension.

Greenfeld and Tao \mbox{\cite[Question~10.4]{GTCounterexample}} asked
how few prime factors the cardinality of an aperiodic tile can have.
Throughout this paper, prime factors are counted with multiplicity.
Here, for an integer $d\geq1$ and a finite set
$F\subseteq\mathbb Z^d$, we call $F$ an \emph{aperiodic tile} if
$F$ tiles $\mathbb Z^d$ by translations but has no fully periodic
tiling complement. Khetan proved that if $|F|=p^2$ for a prime $p$
and $F$ tiles $\mathbb Z^3$, then it admits a weakly periodic
complement, namely a
finite disjoint union of sets each invariant under a nonzero
translation \cite[Theorem~4.1]{Khetan}.

Our main result establishes full periodicity for two composite
cardinality families, with no restriction on dimension.

\begin{theorem}\label{thm:main}
Let $d\geq1$ and let $F\subseteq\mathbb Z^d$ be finite. Suppose that
either $|F|=2q$ for an odd prime $q$, or $|F|=9$. If $F$ tiles
$\mathbb Z^d$ by translations, then it admits a fully periodic
tiling complement. Equivalently, there exist $M\geq1$ and
$B\subseteq\mathbb Z^d$ such that
\[
F\oplus B=\mathbb Z^d,
\qquad
B+M\mathbb Z^d=B.
\]
\end{theorem}

The theorem imposes no restriction on the affine span or intrinsic
rank of $F$. It covers, in particular, the cardinalities
$6,9,10,14,22,26,\ldots$. Together with the singleton and prime cases
and the known four-point case, it proves periodic tileability for every
positive cardinality below twelve except possibly eight.

The twice-prime family also places the four-point theorem in a natural
sequence. The case $|F|=4$ is the $q=2$ endpoint of the same
arithmetic pattern, although the coincidence of the two prime
factors changes the vanishing-sum algebra. We therefore treat it
separately in Subsection~\ref{sec:four}.

For nine-point tiles, our result should be compared with Khetan's
prime-square theorem in $\mathbb Z^3$ \cite{Khetan}. Khetan obtains
weak periodicity for $p^2$-point tiles in dimension three. At $p=3$,
we upgrade this to full periodicity and remove the restriction on
dimension. Thus aperiodic tiles cannot have cardinality $2q$ or $9$;
among cardinalities with exactly two prime factors, counted with
multiplicity, the unresolved families are products of two distinct odd
primes and squares $p^2$ with $p\geq5$.

\subsection{Outline of the proof}

The proofs for $|F|=2q$ and $|F|=9$ begin from the same algebraic
principle and diverge when the possible vanishing sums of roots of
unity are analyzed.

Let $d\geq1$ be an integer, let $F\subseteq\mathbb Z^d$ be finite, and
let $A\subseteq\mathbb Z^d$ satisfy $F\oplus A=\mathbb Z^d$. Define the
associated tiling indicator $a:=1_A\colon\mathbb Z^d\to\{0,1\}$ by
\[
a(x)=1_A(x)=
\begin{cases}
1,&x\in A,\\
0,&x\notin A.
\end{cases}
\]
The coprime dilation identity is $1_{sF}*a=1$ for every positive
integer $s$ with $\gcd(s,|F|)=1$; see~\cite{GGR,HorakKim,Szegedy}.
On the spectral side, these identities
constrain the characters in the support of the centered tiling
observable. The character values on $F$ split into zero-sum blocks
in which all phase ratios are roots of unity of a prescribed order.
After normalizing each block by one of its phases, we can apply the
arithmetic of vanishing sums of roots of unity.

For both cardinality families, we translate the tile and restrict the
given tiling to its intrinsic lattice. A periodic complement obtained
there extends to the ambient lattice. In the nine-point case, intrinsic
ranks $1$ and $2$ are handled by Newman and Bhattacharya, respectively,
before the higher-rank argument.

\smallskip\noindent
\emph{Twice-prime cardinality.}
For $|F|=2q$, with $q$ an odd prime, the relevant vanishing sums
decompose into either $q$ opposite pairs or two $q$-cycles, by the
two-prime theory of vanishing sums~\cite{deBruijn,LamLeung,PoonenRubinstein}.
In the pair case, the
multiplier $P_q(\theta)=\sum_{u\in F}\theta(qu)$ annihilates the
corresponding spectral contribution. In the two-cycle case, the
difference structure either yields a lattice complement or forces the
filtered observable to have finite torsion spectrum.

The Frobenius congruence makes $c=q^{-1}1_{qF}*a$ an
integer-valued twofold covering. The spectral analysis makes this
covering periodic for a suitable tiling indicator $a$. Moreover,
each translate $b(x)=a(x-qf_0)$, $f_0\in F$, vanishes where
$c=0$ and equals one where $c=2$. This support compatibility
provides a proper two-coloring on the multiplicity-one sites,
which can be replaced by a periodic coloring to obtain a
periodic tiling complement. The coloring lemma is isolated in
Section~\ref{sec:split} for later use.

\smallskip\noindent
\emph{Nine-point cardinality.}
For $|F|=9$, the dilation equations partition the character values
on $F$ into zero-sum blocks whose sizes are divisible by three.
There are therefore at most three blocks. The subgroups generated by
within-block differences determine the geometry of the spectral
support. If the subgroup generated by within-block differences has
rank two less than the intrinsic rank, $F$ projects bijectively onto
$(\mathbb Z/3\mathbb Z)^2$, which gives a lattice complement.
Otherwise the spectral measure is supported on finitely many rational
affine circles in the dual torus.

For an integer $d\geq1$, a \emph{rational affine circle} in
$\widehat{\mathbb Z^d}$ is a torsion translate of a connected
one-dimensional subtorus of $\widehat{\mathbb Z^d}$.

Let $\varnothing\ne F\subseteq\mathbb Z^d$ be finite. Define
\[
X_F:=\{a\in\{0,1\}^{\mathbb Z^d}:
       a\text{ is a tiling indicator of }F\}.
\]
Give $\{0,1\}^{\mathbb Z^d}$ the product topology, with $\{0,1\}$
discrete, and give $X_F$ the inherited subspace topology.
Equivalently, $a\in X_F$ exactly when
$\sum_{f\in F}a(x-f)=1$ for every $x\in\mathbb Z^d$. Each such
equation involves only finitely many coordinates, so $X_F$ is closed
in the compact product $\{0,1\}^{\mathbb Z^d}$ and hence is compact.
Let $G=\mathbb Z^d$ act on $X_F$ by the shifts
\[
 (T_ga)(x)=a(x+g).
\]
Define the origin-coordinate event by
\[
 D:=\{a\in X_F:a(0)=1\}.
\]
The coordinate map $a\mapsto a(0)$ is continuous, so $D$ is clopen in
$X_F$ and hence is Borel measurable. For an invariant Borel probability
measure $\mu$ on $X_F$, the Koopman convention is
$U_g\phi=\phi\circ T_{-g}$. Invariance and ergodicity below refer to
this shift action, and spectral measures refer to this Koopman
representation.

\begin{samepage}

\begin{theorem}[Spectral-circle periodicization]\label{thm:circle}
Let $F\subseteq\mathbb Z^d$ be nonempty and finite, with $d\geq2$.
Suppose the shift action on $X_F$ admits an ergodic invariant
probability measure $\mu$ for which the spectral measure of $1_D$ is
supported on finitely many rational affine circles. Then $F$ admits a fully
periodic tiling complement.
\end{theorem}

\end{samepage}

We prove Theorem~\ref{thm:circle} by extending the conditional-density
arguments of Bhattacharya \mbox{\cite[Theorem~4.4]{Bhattacharya}}
and Khetan \cite[Appendix~C]{Khetan}. After passing to a suitable
finite-index
subgroup, all measurable invariance is understood modulo null sets.
Call an ergodic component $E$ ordinary when $D\cap E$ is piecewise
$(d-1)$-periodic and exceptional otherwise. The conditional-density
dichotomy implies that every exceptional component has density $1/2$;
an ordinary component may also have that density. The tiling partition
ensures that each component has either zero or two exceptional
translates. A periodic two-coloring replaces every such pair
simultaneously and produces a piecewise $(d-1)$-periodic tiling
complement.

Finally, the periodicization theorem of Meyerovitch, Sanadhya, and
Solomon \cite[Theorem~1.3]{MSS} converts such a piecewise
$(d-1)$-periodic co-tile into a fully periodic one.

The rest of this paper is organized as follows. 
Section~\ref{sec:algebra} develops the common algebraic, dynamical,
and spectral preliminaries. Section~\ref{sec:split} proves the
periodic coloring lemma. Section~\ref{sec:twice-prime} treats the
$2q$ family, followed by the four-point endpoint in
Subsection~\ref{sec:four}. Section~\ref{sec:circles} proves
Theorem~\ref{thm:circle}, and Section~\ref{sec:nine} applies it to
nine-point tiles. Finally, Section~\ref{sec:effective} derives period
reduction and decidability consequences for both cardinality
families.

\section{Algebraic and spectral preliminaries}\label{sec:algebra}

We first reduce to the lattice generated by the tile. We then collect
the spectral identities and dilation lemmas used in both cardinality
arguments.

\subsection{The intrinsic lattice and periodicity}

All groups below are written additively. Let $R$ be a commutative ring.
For a finitely supported function $u\colon G\to R$ on an abelian group $G$
and any function $v\colon G\to R$, set, for $x\in G$,
\[
 (u*v)(x)=\sum_{y\in G}u(y)v(x-y).
\]
For indicator functions and tiling equations we take $R=\mathbb Z$:
for a subset $S\subseteq G$, let $1_S\colon G\to\mathbb Z$ be its indicator,
and let $1\colon G\to\mathbb Z$ be the constant function with value one. If
$F\subseteq G$ is finite, then
\[
 (1_F*1_A)(x)=\bigl|\{(f,a)\in F\times A:x=f+a\}\bigr|.
\]
Thus $F\oplus A=G$ is equivalent to $1_F*1_A=1$. A
\emph{lattice complement} for $F$ in a free abelian group $G$ is a
subgroup $L\leq G$ such that every $g\in G$ has a unique expression
$g=f+\ell$ with $f\in F$ and $\ell\in L$; we write this as
$F\oplus L=G$.

\begin{lemma}[Index of a lattice complement]\label{lem:lattice-index}
If $L$ is a lattice complement for a finite set $F\subseteq G$, then
$L$ has finite index in $G$ and
\[
 [G:L]=|F|.
\]
\end{lemma}

\begin{proof}
Let $\pi\colon G\to G/L$ be the quotient map. Existence of the representations
in $F\oplus L=G$ makes $\pi|_F$ surjective. If
$\pi(f)=\pi(f')$, then
\[
 f=f+0=f'+(f-f')
\]
are two $F+L$ representations of $f$. Uniqueness gives $f=f'$.
Hence $\pi|_F$ is a bijection onto $G/L$, proving the claim.
\end{proof}

When $F$ is a finite multiset, $1_F(g)$ records the multiplicity of
$g$. For a positive integer $s$, the dilated multiset $sF$ is the
pushforward under $f\mapsto sf$; thus collisions are retained and
\[
 1_{sF}(g)=\sum_{\substack{f\in G\\sf=g}}1_F(f).
\]
This convention is useful in Lemma~\ref{lem:dilation} even though the
main theorem concerns sets in a torsion-free group.

For $G\cong\mathbb Z^r$ and $0\leq k\leq r$, a set $A\subseteq G$ is
\emph{$k$-periodic} if $A+H=A$ for some rank-$k$ subgroup $H\leq G$,
and \emph{piecewise $k$-periodic} if it is a finite nonempty disjoint
union of $k$-periodic sets. When $k=r$, the set is also called
\emph{fully periodic}. The period groups may differ between pieces.
For a probability-preserving action $T\colon G\curvearrowright(X,\mu)$, a
measurable set $E$ is $k$-periodic if
\[
 \mu(T_hE\mathbin{\triangle}E)=0\qquad(h\in H)
\]
for some rank-$k$ subgroup $H$, and the piecewise notion is defined by
a finite measurable disjoint union. Thus measurable invariance is
always understood modulo null sets.

For a nonempty finite $F\subseteq\mathbb Z^d$, its \emph{intrinsic
lattice} is
\[
 \langle F-F\rangle,
 \qquad F-F=\{f-g:f,g\in F\}.
\]
The following makes the usual normalization and restriction to this
lattice explicit.

\begin{lemma}[Intrinsic-lattice reduction]\label{lem:intrinsic-reduction}
Suppose that $F\oplus A=\mathbb Z^d$. Choose $f_*\in F$ and set
\[
 F'=F-f_*,\qquad A'=A+f_*,\qquad L=\langle F'\rangle.
\]
Then $0\in F'$, $F'\oplus A'=\mathbb Z^d$,
\[
 L=\langle F'-F'\rangle=\langle F-F\rangle,
\]
and
\[
 F'\oplus(A'\cap L)=L.
\]
\end{lemma}

\begin{proof}
The map
\[
 F\times A\longrightarrow F'\times A',\qquad
 (f,a)\longmapsto(f-f_*,a+f_*)
\]
is a bijection that preserves the sum of the coordinates, so it
preserves existence and uniqueness of representations. Since
$f_*\in F$, we have $0\in F'$ and $F'-F'=F-F$. The inclusion
$0\in F'$ also gives
\[
 \langle F'\rangle=\langle F'-F'\rangle=\langle F-F\rangle.
\]
Finally, if $x\in L$ and $x=f'+a'$ is its unique representation in
the translated ambient tiling, then $a'=x-f'\in L$. Thus
$a'\in A'\cap L$, and uniqueness is inherited from the ambient
tiling.
\end{proof}

The next lemma allows a periodic complement constructed in the
intrinsic lattice to be extended to the original ambient lattice.

\begin{lemma}[Extension from a subgroup]\label{lem:extend}
Let $G\leq\mathbb Z^d$ and let $F\subseteq G$ be finite. Suppose that
$F\oplus B=G$ for some $B\subseteq G$, and that there is a subgroup
$P\leq G$ such that $[G:P]<\infty$ and $B+P=B$. Then there are
$A\subseteq\mathbb Z^d$ and a finite-index subgroup
$\Gamma\leq\mathbb Z^d$ such that
\[
 F\oplus A=\mathbb Z^d,
 \qquad A+\Gamma=A.
\]
\end{lemma}

\begin{proof}
Let $b=1_B$. The tiling hypothesis says that
\[
 \sum_{f\in F}b(y-f)=1\qquad(y\in G).
\]
Put
$G^{\mathrm{sat}}=(\operatorname{span}_{\mathbb Q}G)\cap\mathbb Z^d$.
The groups $G^{\mathrm{sat}}/G$ and $G/P$ are finite. Choose positive
integers $e,t$ annihilating them, respectively, so that
$eG^{\mathrm{sat}}\subseteq G$ and $tG\subseteq P$, and set $M=et$.
We have
\begin{equation}\label{eq:ambient-period}
 N:=G\cap M\mathbb Z^d\subseteq P.
\end{equation}
Indeed, if $x=Mz\in G$ with $z\in\mathbb Z^d$, then
$z=x/M\in G^{\mathrm{sat}}$, so $x=t(ez)\in tG\subseteq P$.

Let $\pi\colon\mathbb Z^d\to Q=\mathbb Z^d/M\mathbb Z^d$ and
$H_0=\pi(G)$. Since $N\subseteq P$ and $B+P=B$, the function $b$ is
$N$-periodic and descends to the indicator of $\pi(B)$ on
$H_0\cong G/N$. The map $\pi$ is injective on $F$. Indeed, if
$f_1\ne f_2$ and $f_1-f_2\in N$, choose $x\in B$. At
$y=x+f_1$, the two distinct summands
\[
 b(y-f_1)=b(x)=1,
 \qquad b(y-f_2)=b(x+f_1-f_2)=1
\]
contradict the tiling equation, because $f_1-f_2\in P$.
Consequently the descended equation is the genuine set tiling
\[
 \pi(F)\oplus\pi(B)=H_0.
\]

Choose a set $R\subseteq Q$ of representatives for $Q/H_0$. The
tilings of the different cosets are disjoint, giving
$\pi(F)\oplus(\pi(B)+R)=Q$. Put $C=\pi(B)+R$ and
$A=\pi^{-1}(C)$. For every $x\in\mathbb Z^d$,
\[
 \sum_{f\in F}1_A(x-f)
 =\sum_{f\in F}1_C(\pi(x)-\pi(f))
 =\sum_{u\in\pi(F)}1_C(\pi(x)-u)=1.
\]
Thus $F\oplus A=\mathbb Z^d$. Finally,
$A+M\mathbb Z^d=A$, and $M\mathbb Z^d$ has finite index in
$\mathbb Z^d$.
\end{proof}

\subsection{Tiling systems and spectral measures}\label{subsec:systems}

In the spectral arguments, $G\cong\mathbb Z^r$, with $r\geq0$, is a
finite-rank free abelian group endowed with the discrete topology, and
$\varnothing\ne F\subseteq G$ is finite. The dual group is
\[
 \widehat G=\operatorname{Hom}\bigl(G,\{z\in\mathbb C:|z|=1\}\bigr),
\]
with pointwise multiplication and the compact-open topology,
equivalently the subspace topology inherited from the product of unit
circles. Its identity is the trivial character $\mathbf1$. A character
is \emph{torsion} if it has finite order. For a subgroup $H\leq G$,
its annihilator is
\[
 H^\perp=\{\theta\in\widehat G:\theta(h)=1
                    \text{ for every }h\in H\}.
\]

The tiling space
\[
 X_F=\{a\in\{0,1\}^G:1_F*a=1\}
\]
is a closed subset of the compact product $\{0,1\}^G$, because each
tiling equation involves only finitely many coordinates. We use the
shift
\[
 (T_ga)(x)=a(x+g)
\]
and, for every invariant Borel probability measure $\mu$ on $X_F$,
the unitary operators
\[
 U_g\phi=\phi\circ T_{-g}.
\]

If $a_0\in X_F$, identify $G$ with $\mathbb Z^r$, put
$Q_N=[-N,N]^r\cap\mathbb Z^r$, and set
\[
 \nu_N=|Q_N|^{-1}\sum_{x\in Q_N}\delta_{T_xa_0}.
\]
Since $X_F$ is compact metrizable, the space of Borel probability
measures on it is weakly compact, so $(\nu_N)$ has a weakly convergent
subsequence. For $\phi\in C(X_F)$ and fixed
$g=(g_1,\ldots,g_r)\in G$,
\[
 \left|\int\phi\circ T_g\,d\nu_N-\int\phi\,d\nu_N\right|
 \leq\|\phi\|_\infty
       \frac{|(Q_N+g)\mathbin{\triangle}Q_N|}{|Q_N|}
 \leq\|\phi\|_\infty
       \frac{2\sum_{i=1}^r|g_i|}{2N+1}
 \longrightarrow0
\]
for all sufficiently large $N$. Every weak subsequential limit is
therefore invariant. Applied to any such limit, the standard
ergodic-decomposition theorem for this countable action
\cite[Theorem~3.1]{Chen} yields an ergodic invariant Borel probability
measure on $X_F$.

The following definitions and coordinate identities apply in every
rank, including $r=0$. Define the origin-coordinate event and observable by
\[
 D=\{a\in X_F:a(0)=1\},
 \qquad f(a)=a(0)=1_D(a).
\]

\begin{lemma}[Coordinate identities]\label{lem:coordinate-identities}
For every invariant Borel probability measure $\mu$ on $X_F$, one has
\[
 \sum_{u\in F}U_uf=1\qquad\text{pointwise on }X_F,
\]
\[
 X_F=\bigsqcup_{u\in F}T_uD,
\]
and
\[
 \int_{X_F}f\,d\mu=\frac1{|F|}.
\]
\end{lemma}

\begin{proof}
For $a\in X_F$ and $u\in F$,
\[
 (U_uf)(a)=f(T_{-u}a)=a(-u).
\]
The tiling equation at the origin gives
$\sum_{u\in F}a(-u)=(1_F*a)(0)=1$, proving the first identity. Also,
\[
 a\in T_uD\quad\Longleftrightarrow\quad
 T_{-u}a\in D\quad\Longleftrightarrow\quad a(-u)=1.
\]
The binary values $a(-u)$ sum to one, so exactly one of these events
occurs. Finally, invariance gives
$\int U_uf\,d\mu=\int f\,d\mu$ for every $u$; integrating the first
identity yields the asserted mean.
\end{proof}

More generally, let $T\colon G\curvearrowright(X,\mathcal B,\mu)$ be a
measurable probability-preserving action on a standard probability
space, and write $U_g\phi=\phi\circ T_{-g}$. We use the complex Hilbert
space $L^2(X,\mu)$ with inner product conjugate-linear in the first
variable:
\[
 \langle\phi,\psi\rangle=\int_X\overline\phi\,\psi\,d\mu.
\]
The spectral theorem assigns to each $v\in L^2(X,\mu)$ a finite
positive measure $\sigma_v$ on $\widehat G$, characterized by
\[
 \langle v,U_gv\rangle
   =\int_{\widehat G}\theta(g)\,d\sigma_v(\theta).
\]
When $r=0$, so that $G=\{0\}$ and $\widehat G=\{\mathbf1\}$, this means
explicitly
\[
 \sigma_v=\|v\|_2^2\delta_{\mathbf1}.
\]
When $r\geq1$, the usual spectral theorem for unitary representations
of locally compact second-countable abelian groups applies: $G$ is
discrete and countable, and $L^2(X,\mu)$ is separable because the
probability space is standard. Thus, in every rank, for any finitely
supported complex coefficients $\alpha_g$,
\begin{equation}\label{eq:spectral-norm}
 \left\|\sum_g\alpha_gU_gv\right\|_2^2
 =\int_{\widehat G}\left|\sum_g\alpha_g\theta(g)\right|^2
                  d\sigma_v(\theta).
\end{equation}
If $w=\sum_g\alpha_gU_gv$, then
\[
 d\sigma_w=\left|\sum_g\alpha_g\theta(g)\right|^2d\sigma_v.
\]
These identities can also be formulated through positive-definite
functions; see~\cite{Rudin}.

For $L\leq G$, let
\[
 \mathcal I_L=
 \{E:\mu(T_\ell^{-1}E\mathbin{\triangle}E)=0
                         \text{ for every }\ell\in L\}
\]
be the invariant $\sigma$-algebra, and put
\[
 P_Lv=\mathbb E(v\mid\mathcal I_L).
\]
Conditional expectation is the orthogonal projection onto
\[
 L^2(X,\mathcal I_L,\mu)
 =\{h\in L^2(X,\mu):U_\ell h=h
                         \text{ for every }\ell\in L\}.
\]
All identities here are modulo null sets.

For $v\in L^2(X,\mu)$, write
\[
 \mathcal H_v=\overline{\operatorname{span}}
                 \{U_gv:g\in G\}.
\]
For a finite positive measure $\nu$ on $\widehat G$, let
$(M_g\psi)(\theta)=\theta(g)\psi(\theta)$. A \emph{cyclic spectral
multiplier model} for $(U,v)$ is a unitary
$W_v\colon\mathcal H_v\to L^2(\widehat G,\nu)$ such that
$W_vv=1$ and $W_vU_gW_v^{-1}=M_g$.

\begin{proposition}[Invariant projection in a cyclic model]
\label{prop:invariant-projection}
There is a cyclic spectral multiplier model $(\sigma_v,W_v)$ for
$(U,v)$. If $w\in\mathcal H_v$ and $m=W_vw$, then
\[
 d\sigma_w=|m|^2\,d\sigma_v.
\]
For every subgroup $L\leq G$,
\[
 P_L\mathcal H_v\subseteq\mathcal H_v,
 \qquad W_v(P_Lw)=1_{L^\perp}m,
\]
and consequently
\[
 d\sigma_{P_Lw}=1_{L^\perp}\,d\sigma_w.
\]
In particular,
\[
 W_v(P_Lv)=1_{L^\perp},
 \qquad d\sigma_{P_Lv}=1_{L^\perp}\,d\sigma_v.
\]
\end{proposition}

\begin{proof}
On the algebraic cyclic span define
\[
 W_v^0\!\left(\sum_g\alpha_gU_gv\right)(\theta)
       =\sum_g\alpha_g\theta(g).
\]
Equation~\eqref{eq:spectral-norm} shows that $W_v^0$ is well defined
and isometric. Its range is the trigonometric polynomials on the
compact group $\widehat G$. These are uniformly dense in
$C(\widehat G)$ by the Stone--Weierstrass theorem, and continuous
functions are dense in $L^2(\widehat G,\sigma_v)$ by regularity of the
finite Borel measure. Hence $W_v^0$ extends to a unitary
\[
 W_v\colon\mathcal H_v\longrightarrow L^2(\widehat G,\sigma_v).
\]
Direct calculation gives $W_vv=1$ and $W_vU_g=M_gW_v$.

If $w\in\mathcal H_v$ and $m=W_vw$, then, for every $g\in G$,
\[
 \langle w,U_gw\rangle
 =\int_{\widehat G}\theta(g)|m(\theta)|^2\,d\sigma_v(\theta).
\]
Uniqueness of a finite measure from its character integrals gives
$d\sigma_w=|m|^2d\sigma_v$.

A multiplier vector $m$ is fixed by every $M_\ell$, $\ell\in L$,
exactly when
\[
 (\theta(\ell)-1)m(\theta)=0
 \quad\text{for almost every $\theta$ and every $\ell\in L$}.
\]
Since $L$ is countable and
\[
 \widehat G\setminus L^\perp
 =\bigcup_{\ell\in L}\{\theta:\theta(\ell)\ne1\},
\]
the fixed subspace is precisely $L^2(L^\perp,\sigma_v)$, whose
orthogonal projection is multiplication by $1_{L^\perp}$. Moreover,
$\mathcal H_v$ reduces every $U_g$, since it is invariant under both
$U_g$ and $U_g^*=U_{-g}$. The global fixed-space projection $P_L$
therefore preserves $\mathcal H_v$ and has the stated multiplier
form there. The remaining formulas follow from the spectral-density
identity just proved.
\end{proof}

\begin{lemma}[Bounds for invariant projections]\label{lem:invariant-bounds}
Let $L\leq G$. If $f\in L^2(X,\mu)$ is real-valued and
$0\leq f\leq1$ almost everywhere, then
\[
 0\leq P_Lf\leq1\qquad\text{almost everywhere}.
\]
\end{lemma}

\begin{proof}
Put $h=P_Lf$. It has a real-valued $\mathcal I_L$-measurable
representative and satisfies
\[
 \int_Eh\,d\mu=\int_Ef\,d\mu\qquad(E\in\mathcal I_L).
\]
For $n\geq1$, let $A_n=\{h\leq-1/n\}$. If $\mu(A_n)>0$, then
\[
 0\leq\int_{A_n}f\,d\mu
 =\int_{A_n}h\,d\mu
 \leq-\frac{\mu(A_n)}n<0,
\]
a contradiction. Hence $h\geq0$ almost everywhere. Similarly, for
$B_n=\{h\geq1+1/n\}$, positive measure would imply
\[
 \left(1+\frac1n\right)\mu(B_n)
 \leq\int_{B_n}h\,d\mu
 =\int_{B_n}f\,d\mu
 \leq\mu(B_n),
\]
which is impossible. Thus $h\leq1$ almost everywhere.
\end{proof}

Two consequences will be used repeatedly. First, if
$\operatorname{supp}\sigma_v\subseteq S$ for a finite set of torsion
characters, put $\Gamma=\bigcap_{\theta\in S}\ker\theta$. If
$S\ne\varnothing$ and $N$ is the least common multiple of the orders
of its characters, then $NG\leq\Gamma$; hence $\Gamma$ has finite
index in $G$. For $\gamma\in\Gamma$,
\begin{equation}\label{eq:torsion-periods}
 \|(U_\gamma-I)v\|_2^2
 =\int_S|\theta(\gamma)-1|^2\,d\sigma_v(\theta)=0.
\end{equation}
Thus $v$ is $\Gamma$-invariant in $L^2$; for $S=\varnothing$, take
$\Gamma=G$.

Second, write $\mathbb T=\mathbb R/\mathbb Z$. Smith normal form gives
a basis $e_1,\ldots,e_r$ of $G$ and positive integers
$m_1\mid m_2\mid\cdots\mid m_k$ such that, for a rank-$k$ subgroup
$L\leq G$,
\[
 L=\bigoplus_{i=1}^k m_i\mathbb Ze_i.
\]
Indeed, take a basis of the free abelian group $L$ and form the
$r\times r$ integer matrix whose first $k$ columns are its coordinate
vectors in a fixed basis of $G$ and whose remaining columns are zero.
Applying Smith normal form~\cite[Theorem~21.4]{DasNg}, column operations
preserve the subgroup generated by the columns, while row operations
amount to changing the basis of $G$. In the corresponding dual
coordinates, the map
\[
 \Phi\colon\mathbb T^r\longrightarrow\widehat G,
 \qquad
 \Phi(x_1,\ldots,x_r)\!\left(\sum_{i=1}^ra_ie_i\right)
 =\exp\!\left(2\pi i\sum_{i=1}^ra_ix_i\right)
\]
is an isomorphism of topological groups, and
\begin{equation}\label{eq:annihilator-coordinates}
 L^\perp=
 \prod_{i=1}^k\bigl(m_i^{-1}\mathbb Z/\mathbb Z\bigr)
 \times\mathbb T^{r-k}.
\end{equation}
For $k=r$, factorization through $G/L$ identifies $L^\perp$ with the
character group of $G/L$; it is a finite torsion group of order
\[
 |L^\perp|=[G:L]=\prod_{i=1}^rm_i.
\]
For $k=r-1$, formula~\eqref{eq:annihilator-coordinates} is a finite
union of torsion translates of a one-dimensional rational subtorus,
hence a finite union of rational affine circles. When $r\geq2$, each
torsion point lies on a rational affine circle, namely a translate of
$\{0\}^{r-1}\times\mathbb T$ by that point in suitable coordinates.

\begin{lemma}\label{lem:components}
Let an abelian group $G$ act ergodically by probability-preserving
transformations on $(X,\mu)$, and let $\Gamma\leq G$ have finite
index. Up to null sets, $X$ is the disjoint union of finitely many
positive-measure $\Gamma$-invariant sets on each of which the
restricted $\Gamma$-action is ergodic. These components are permuted
transitively by $G$ and have equal measure.
\end{lemma}

\begin{proof}
Let $m=[G:\Gamma]$ and choose representatives $R$ for $G/\Gamma$,
with $0\in R$. If $B$ is $\Gamma$-invariant, then commutativity shows
that $G$ permutes the family $(T_rB)_{r\in R}$. Hence
\[
 N_B=\sum_{r\in R}1_{T_rB}
\]
is $G$-invariant and therefore constant almost everywhere. Since the
action preserves $\mu$, its value is
$\int N_B\,d\mu=m\mu(B)$. If $\mu(B)>0$, then $N_B\geq1$ on $B$,
so $\mu(B)\geq1/m$.

Consequently the $\Gamma$-invariant $\sigma$-algebra has at most $m$
pairwise disjoint positive-measure members. Starting with $X$, split
any invariant piece that is not an atom into two positive-measure
invariant pieces. The preceding bound forces this procedure to stop
after at most $m-1$ splittings. It produces a partition into at most
$m$ atoms of the invariant $\sigma$-algebra, and on each atom the
restricted $\Gamma$-action is ergodic.

For $g\in G$, commutativity makes $T_g$ carry each such atom to another
one. The translates of any fixed atom cover $X$ up to a null set:
indeed, applying the first paragraph to that atom makes the
corresponding $N_B$ a positive constant. Thus every component is a
translate of every other, proving transitivity. Measure preservation
then gives equal masses.
\end{proof}

For a positive-measure $H$-ergodic component $E$, write
\[
 \mu_E(B)=\frac{\mu(B\cap E)}{\mu(E)}.
\]
Commutativity and measure preservation show that $T_gE$ is again an
$H$-ergodic component, and consequently
\[
 \mu_{T_gE}(T_gB)=\mu_E(B)
\]
for every measurable $B$ and every $g\in G$.

\subsection{Coprime dilation and vanishing sums}

The following dilation identity converts the tiling equation into
simultaneous annihilation equations. We use the form in
\cite[Theorem~8]{HorakKim} and \cite[Theorem~1.2(i)]{GGR};
see also~\cite{GTStructure}.

\begin{lemma}[Coprime dilation]\label{lem:dilation}
Let $F$ be a nonempty finite multiset in an abelian group $G$, of total
cardinality $n$. If $a\colon G\to\{0,1\}$ satisfies $1_F*a=1$, then
$1_{sF}*a=1$ for every positive integer $s$ with $\gcd(s,n)=1$.
\end{lemma}

\begin{proof}
It suffices first to consider a prime $\ell$ not dividing $n$.
The group ring over $\mathbb Z/\ell\mathbb Z$ is commutative, so
the Frobenius identity gives
$1_{\ell F}\equiv1_F^{*\ell}\pmod\ell$.
The congruence remains true with repeated points, since an integer
multiplicity $m$ satisfies $m^\ell\equiv m\pmod\ell$.
Consequently the nonnegative integer-valued function
$b=1_{\ell F}*a$ satisfies
\begin{equation}\label{eq:dilate-congruence}
 b\equiv1_F^{*\ell}*a
   =1_F^{*(\ell-1)}*1=n^{\ell-1}\equiv1\pmod\ell.
\end{equation}
Hence $b(x)\geq1$ for every $x$. Commuting the finite convolutions,
we also have $1_F*b=1_{\ell F}*(1_F*a)=1_{\ell F}*1=n$.
Each value of $1_F*b$ is a sum of $n$ integers at least one, counted
with multiplicity, so every summand equals one. Since $F$ is nonempty,
every value $b(x)$ occurs in such a sum. Therefore $b=1$.

For $s=1$, the conclusion is the hypothesis. For $s>1$, factor $s$
into primes and apply the prime step successively to the resulting
dilated multisets. Each has total cardinality $n$, and each prime
factor is coprime to $n$.

\end{proof}

For every positive integer $m$, let
$\mu_m=\{z\in\mathbb C:z^m=1\}$. A \emph{rotated $m$-cycle} is the
multiset $z\mu_m$ for some $|z|=1$, with each of its $m$ elements
occurring once.

We will also need the following special case of the two-prime
classification of vanishing sums. It follows from
\cite[Theorem~3.3]{LamLeung}, with antecedents in
\cite[\S3]{deBruijn}; the $2q$ case is also stated in
\cite[Lemma~2]{PoonenRubinstein}. We include an elementary proof.

\begin{samepage}
\begin{lemma}[Vanishing sums with two prime factors]\label{lem:roots}
Let $q$ be an odd prime. Every nonnegative integral vanishing sum
of elements of $\mu_{2q}$ is a disjoint multiset union of opposite
pairs and rotated $q$-cycles. The same assertion holds after rotating
all the terms by a common unit complex number.

If a collection of such rotated vanishing sums has total length $2q$,
then its terms admit a partition into either $q$ opposite pairs or
two rotated $q$-cycles.
\end{lemma}
\end{samepage}

\begin{proof}
Put $\zeta=\exp(2\pi i/q)$. Since $q$ is odd, every element of
$\mu_{2q}$ is uniquely $\zeta^j$ or $-\zeta^j$, with
$0\leq j<q$. Write the vanishing sum as
$\sum_{j=0}^{q-1}(a_j-b_j)\zeta^j=0$, where $a_j,b_j$ are
nonnegative integers. The polynomial
$D(X)=\sum_{j=0}^{q-1}(a_j-b_j)X^j$ is divisible over
$\mathbb Q[X]$ by the minimal polynomial
$\Phi_q(X)=1+X+\cdots+X^{q-1}$. Since $\deg D\leq q-1$,
there is an integer $t$ such that $a_j-b_j=t$ for every $j$.

Remove $\min(a_j,b_j)$ opposite pairs at each $j$. If $t\geq0$,
the remaining terms are $t$ copies of $\mu_q$; if $t<0$, they are
$-t$ copies of $-\mu_q$. Dividing a rotated sum by its common
rotation proves the rotated version as well.

Apply this decomposition separately to every sum in the collection.
If there are $u$ opposite pairs and $v$ cycles in total, then
$2u+qv=2q$. Since $q$ is odd, $v$ must be even, and since
$qv\leq2q$, the only possibilities are $(u,v)=(q,0)$ and
$(u,v)=(0,2)$.
\end{proof}

\subsection{Common zeros and finite quotients}

The dilation equations give a finite collection of spectral
constraints. We first record their common spectral consequence. The
two subsequent lemmas extract the resulting block structure and
identify when that structure already yields a lattice complement.

Let $G$ be a finite-rank free abelian group and let $F\subseteq G$ be
a nonempty finite set of cardinality $n$. For a positive integer $s$,
put
\[
 P_s(\theta)=\sum_{u\in F}\theta(su),
 \qquad P_s(U)=\sum_{u\in F}U_{su}.
\]
For a positive integer $Q$ divisible by every prime divisor of $n$,
define
\begin{equation}\label{eq:general-zeros}
 Z_{F,Q}=\bigcap_{j=0}^{n-1}\{\theta\in\widehat G:
                                      P_{1+Qj}(\theta)=0\}.
\end{equation}

\begin{lemma}[Common-zero spectral support]\label{lem:common-support}
Let $\mu$ be a $G$-invariant Borel probability measure on $X_F$,
let $f(a)=a(0)$, and put $h=f-1/n$. Then
\begin{equation}\label{eq:general-support}
 \operatorname{supp}\sigma_h\subseteq Z_{F,Q},\qquad
 \operatorname{supp}\sigma_f
       \subseteq Z_{F,Q}\cup\{\mathbf1\}.
\end{equation}
\end{lemma}

\begin{proof}
For $0\leq j<n$, put $s_j=1+Qj$. If a prime $p$ divided both $s_j$
and $n$, then $p\mid Q$ and therefore $s_j\equiv1\pmod p$, a
contradiction. Hence $\gcd(s_j,n)=1$. Lemma~\ref{lem:dilation},
applied to the tiling equation of $a\in X_F$, gives
\[
 (P_{s_j}(U)f)(a)
 =\sum_{u\in F}a(-s_ju)
 =(1_{s_jF}*a)(0)=1.
\]
Every $U_g$ fixes constants, so $P_{s_j}(U)(1/n)=1$ and
$P_{s_j}(U)h=0$. Equation~\eqref{eq:spectral-norm} now gives
\[
 0=\|P_{s_j}(U)h\|_2^2
  =\int_{\widehat G}|P_{s_j}(\theta)|^2\,d\sigma_h(\theta).
\]
Because $P_{s_j}$ is continuous, the support of $\sigma_h$ lies in
its zero set. Intersecting over $j$ proves the first inclusion.

Taking $j=0$ in the pointwise identity above and integrating gives
$n\int f\,d\mu=1$, so $h\perp1$. If $G$ has rank zero, then
$G=F=\{0\}$, $n=1$, $f=1$, and $h=0$; hence
\[
 \sigma_f=\delta_{\mathbf1},\qquad \sigma_h=0.
\]
If $G$ has positive rank, the line of constant functions and its
orthogonal complement reduce every $U_g$. The orthogonal spectral
decomposition $f=h+(1/n)1$ therefore gives
\[
 \sigma_f=\sigma_h+\frac1{n^2}\delta_{\mathbf1}.
\]
The same identity holds in rank zero, and the second inclusion
follows.
\end{proof}

\begin{lemma}[Zero-sum blocks]\label{lem:blocks}
Assume that $0\in F$ and $G=\langle F\rangle\cong\mathbb Z^r$.
For $\theta\in Z_{F,Q}$, partition $F$ by equality of
$\theta(u)^Q$. Each block has zero sum of its phases $\theta(u)$.
If there are $b$ blocks and $H$ is generated by differences within
blocks, then $\operatorname{rank}H\geq r-b+1$ and
$\theta\in(QH)^\perp$.
\end{lemma}
\begin{proof}
Write the blocks as $B_1,\ldots,B_b$, and let
$t_i=\theta(u)^Q$ for $u\in B_i$ and
$s_i=\sum_{u\in B_i}\theta(u)$. For $0\leq j<n$,
\[
 0=P_{1+Qj}(\theta)
 =\sum_{i=1}^b\sum_{u\in B_i}\theta(u)\theta(u)^{Qj}
 =\sum_{i=1}^b t_i^j s_i.
\]
Since $b\leq n$, the first $b$ equations give $Vs=0$ with
$V=(t_i^j)_{0\leq j<b,\,1\leq i\leq b}$. The values $t_i$
are distinct, so $\det V=\prod_{i<k}(t_k-t_i)\ne0$.
Thus $s_i=0$ for every block.

Relabel the blocks so that $0\in B_1$, and choose base points
$a_i\in B_i$, with $a_1=0$. Every $u\in B_i$
satisfies $u=(u-a_i)+a_i$, and hence
$G=H+\sum_{i=2}^b\mathbb Za_i$. Taking ranks gives
$r\leq\operatorname{rank}H+b-1$.
For $u,v$ in the same block,
$\theta(Q(u-v))=\theta(u)^Q/\theta(v)^Q=1$.
These differences generate $H$, so $\theta\in(QH)^\perp$.
\end{proof}

\begin{lemma}[Independent blocks]\label{lem:quotient}
Suppose $0\in F\subseteq G=\langle F\rangle\cong\mathbb Z^r$.
Partition $F$ into $b$ blocks $B_0,\ldots,B_{b-1}$ with base points
$a_0=0,a_1,\ldots,a_{b-1}$, and let $H$ be generated by their
within-block differences. Assume $\operatorname{rank}H=r-b+1$.
If a homomorphism $\chi\colon H\to K$ to a finite abelian group maps
each normalized block $B_j-a_j$ bijectively onto $K$, then $F$
has a lattice complement in $G$.
\end{lemma}

\begin{proof}

\emph{Step 1: an integral direct sum.}
Every element of $F$ lies in $H+\sum_{j=1}^{b-1}\mathbb Za_j$,
so the homomorphism
$\rho\colon\mathbb Z^{b-1}\to G/H$,
$\rho(m_1,\ldots,m_{b-1})=\sum_jm_ja_j+H$, is surjective.
Consider the exact sequences
\[
 0\longrightarrow H\longrightarrow G\longrightarrow G/H
 \longrightarrow0
\]
and
\[
 0\longrightarrow\ker\rho\longrightarrow\mathbb Z^{b-1}
 \xrightarrow{\rho}G/H\longrightarrow0.
\]
Tensoring with $\mathbb Q$ preserves exactness, and vector-space
rank--nullity therefore gives
\[
 \operatorname{rank}(G/H)=r-\operatorname{rank}H=b-1,
 \qquad
 \operatorname{rank}\ker\rho=(b-1)-\operatorname{rank}(G/H)=0.
\]
A subgroup of the free abelian group $\mathbb Z^{b-1}$ is free, so
rank zero forces $\ker\rho=0$. Hence $\rho$ is an isomorphism and
$G/H\cong\mathbb Z^{b-1}$. Every $g\in G$ has a unique expression
$g=h+\sum_{j=1}^{b-1}m_ja_j$ with $h\in H$ and $m_j\in\mathbb Z$.

\smallskip\noindent\emph{Step 2: a quotient tiled by the blocks.}
The unique expression makes the following homomorphism well-defined:
\begin{equation}\label{eq:independent-quotient}
 \Psi\left(h+\sum_{j=1}^{b-1}m_ja_j\right)
 =\left(\chi(h),\sum_{j=1}^{b-1}jm_j\bmod b\right)
 \in K\times\mathbb Z/b\mathbb Z.
\end{equation}
For $u\in B_j$, this gives $\Psi(u)=(\chi(u-a_j),j)$.
The assumed bijectivity on $B_j-a_j$ shows that $B_j$ maps
bijectively onto $K\times\{j\}$. Hence $\Psi|_F$ is a
bijection onto the whole quotient, and
$[G:\ker\Psi]=b|K|=|F|$.

For any $g\in G$, there is exactly one $u\in F$ with
$\Psi(u)=\Psi(g)$. Then $g-u\in\ker\Psi$; conversely any
representation $g=u+\ell$, $\ell\in\ker\Psi$, must use this
same $u$. Therefore $F\oplus\ker\Psi=G$.

\end{proof}

\section{Periodic two-colorings and covering splittings}\label{sec:split}

Both cardinality arguments eventually reduce to complementary binary
values on a periodic graph. We first prove that such a graph admits
a periodic proper two-coloring, then apply the result to twofold
coverings.

\subsection{Periodic two-colorings}

For connected quasi-transitive bipartite graphs, periodic
two-colorability is recorded in~\cite[Lemma~2.9]{PeriodicGraphs}.
We need the following lattice version, which also handles
disconnected graphs and prescribes the ambient period subgroup.

\begin{lemma}[Periodic two-colorings]\label{lem:color}
Let $G\cong\mathbb Z^r$ and $V\subseteq G$. Let $\mathcal G$ be
an undirected graph with vertex set $V$. Suppose that $V$ and the
edge set are invariant under a finite-index subgroup $\Gamma\leq G$.
If $\mathcal G$ admits a proper two-coloring, then it admits one
invariant under a finite-index subgroup of $\Gamma$.
\end{lemma}

\begin{proof}
\emph{Step 1: the action of a component stabilizer.}
If $V=\varnothing$, the empty coloring is $\Gamma$-invariant.
Assume henceforth that $V\ne\varnothing$. Since $V/\Gamma$ is finite, so is
the set of $\Gamma$-orbits of connected components: every component
contains a translate of a vertex representative. Choose components
$D_1,\ldots,D_t$ representing these orbits, and set
\begin{equation}\label{eq:stabilizer}
 L_D=\{\gamma\in\Gamma:D+\gamma=D\},\qquad
 L_D^{\mathrm{sat}}=(L_D\otimes\mathbb Q)\cap\Gamma.
\end{equation}

Fix a proper coloring $\kappa_D\colon D\to\mathbb Z/2\mathbb Z$.
A connected bipartite graph, including a single vertex, has exactly
two such colorings. For each $\ell\in L_D$ there is therefore
a unique $\varepsilon_D(\ell)\in\mathbb Z/2\mathbb Z$ with
\[
 \kappa_D(x+\ell)=\kappa_D(x)+\varepsilon_D(\ell)
 \qquad(x\in D).
\]
Applying this identity successively to $\ell$ and $\ell'$ shows
that $\varepsilon_D$ is a homomorphism. Replacing $\kappa_D$
by $\kappa_D+1$ leaves $\varepsilon_D$ unchanged. In particular,
$2L_D\subseteq\ker\varepsilon_D$. Translating the component
leaves $L_D$ and $\varepsilon_D$ unchanged.

\smallskip\noindent\emph{Step 2: a common period subgroup.}
Each $L_D^{\mathrm{sat}}/L_D$ is finite. Choose $m\geq1$
divisible by the exponents of these groups for $D_1,\ldots,D_t$,
and put $K=2m\Gamma$. If $k=2m\gamma\in K\cap L_D$, then
$\gamma\in L_D^{\mathrm{sat}}$, so $m\gamma\in L_D$ and
$k=2(m\gamma)\in2L_D$. Thus
\begin{equation}\label{eq:color-stabilizer}
 K\cap L_D\subseteq\ker\varepsilon_D
 \qquad\text{for every connected component }D.
\end{equation}

\smallskip\noindent\emph{Step 3: construct the periodic coloring.}
Choose representatives $C_1,\ldots,C_u$ for the $K$-orbits of
components; there are finitely many since $[\Gamma:K]<\infty$.
Fix a proper coloring $\kappa_i$ on each $C_i$. For $x\in C_i$
and $k\in K$, define
\[
 \beta(x+k)=\kappa_i(x).
\]
This is well-defined. Indeed, if $x+k=x'+k'$ with
$x,x'\in C_i$, then $\ell=k-k'=x'-x$ belongs to $K\cap L_{C_i}$:
the components $C_i$ and $C_i+\ell$ intersect at $x'$ and hence
coincide. By \eqref{eq:color-stabilizer},
$\kappa_i(x')=\kappa_i(x)+\varepsilon_{C_i}(\ell)=\kappa_i(x)$.
Representations using different $C_i$ would put them in the same
$K$-orbit and are excluded by the choice of representatives.

Every edge lies inside one translated component $C_i+k$, where
$\beta$ is a translate of the proper coloring $\kappa_i$.
Thus $\beta$ is proper. Its definition gives
$\beta(v+k)=\beta(v)$ for all $v\in V$ and $k\in K$.
Since $K$ has finite index in $\Gamma$, this proves the lemma.
\end{proof}

\begin{remark}\label{rem:saturation}
The following example shows that the saturation factor is necessary.
Take $G=V=\Gamma=\mathbb Z$ and join $x$ to $x+2$ for every
$x\in\mathbb Z$. This graph is invariant under all integer
translations, and its connected components are
$D_0=2\mathbb Z$ and $D_1=1+2\mathbb Z$. For $i\in\{0,1\}$,
\[
 L_{D_i}=2\mathbb Z,\qquad
 L_{D_i}^{\mathrm{sat}}
 =(2\mathbb Z\otimes_{\mathbb Z}\mathbb Q)\cap\mathbb Z
 =\mathbb Z.
\]
Thus $L_{D_i}^{\mathrm{sat}}/L_{D_i}\cong\mathbb Z/2\mathbb Z$
has exponent $m=2$, and the construction in the proof gives
$K=2m\Gamma=4\mathbb Z$.

Indeed, define $\kappa\colon\mathbb Z\to\mathbb Z/2\mathbb Z$ by
\[
 \kappa(2n)=\kappa(2n+1)=n\pmod 2
 \qquad(n\in\mathbb Z).
\]
Then $\kappa(x+2)=\kappa(x)+1$ for every $x$, so $\kappa$ is a
proper two-coloring, while $\kappa(x+4)=\kappa(x)$, so it is
$4\mathbb Z$-invariant. Conversely, a coloring invariant under
$2\Gamma=2\mathbb Z$ would satisfy $\beta(x+2)=\beta(x)$ for every
$x$, making every edge monochromatic. Hence a $2\Gamma$-invariant
proper coloring need not exist, and the saturation quotients
$L_D^{\mathrm{sat}}/L_D$ must be accounted for.
\end{remark}

\subsection{Splitting a twofold covering}

Let $G$ be an abelian group and let $F\subseteq G$ be finite.
A function $c\colon G\to\{0,1,2\}$ satisfying $1_F*c=2$, equivalently
\[
 \sum_{f\in F}c(x-f)=2\qquad(x\in G),
\]
will be called a \emph{twofold covering}. The value $c(x)$ is the
multiplicity of the translate $x+F$. Such a covering is
\emph{periodic} if there exists a finite-index subgroup
$\Gamma\leq G$ such that $c(x+\gamma)=c(x)$ for every $x\in G$
and every $\gamma\in\Gamma$.

The next proposition explains the support condition needed to split
such a covering. A tiling complement determines a proper coloring
on the sites where the covering has multiplicity one.

\begin{proposition}[Supported splitting]\label{prop:split}
Let $F\subseteq G\cong\mathbb Z^r$ be a finite set. Suppose that
$c\colon G\to\{0,1,2\}$ is periodic and satisfies $1_F*c=2$.
Suppose also that $b\colon G\to\{0,1\}$ satisfies $1_F*b=1$, vanishes on
$\{c=0\}$, and equals one on $\{c=2\}$. Then $F$ has a fully
periodic tiling complement.
\end{proposition}

\begin{proof}
Let $\Gamma$ be a finite-index period subgroup for $c$, and set
$V=\{x:c(x)=1\}$. For $y\in G$, define
$I_j(y)=\{f\in F:c(y-f)=j\}$, $j=0,1,2$.
The covering identity becomes
\[
 |I_1(y)|+2|I_2(y)|=2.
\]
Thus either $I_2(y)=\{f_0\}$ and $I_1(y)=\varnothing$,
or $I_1(y)=\{f_1,f_2\}$ and $I_2(y)=\varnothing$.

In the second case put an edge between $y-f_1$ and $y-f_2$.
The endpoints are distinct since $F$ is a set. Taking all such
edges gives a graph on $V$. For $\gamma\in\Gamma$,
$I_j(y+\gamma)=I_j(y)$, so translation by $\gamma$ preserves
the graph. The support condition and the tiling equation give
\[
 b(y-f_1)+b(y-f_2)
 =\sum_{f\in F}b(y-f)=1
 \qquad\text{whenever }I_1(y)=\{f_1,f_2\}.
\]
Hence $b|_V$ is a proper two-coloring. By Lemma~\ref{lem:color},
there is a proper two-coloring $\beta\colon V\to\{0,1\}$ with a
finite-index period subgroup $K\leq\Gamma$.

Define the replacement explicitly by
\[
 b'(x)=
 \begin{cases}
  0,&c(x)=0,\\
  \beta(x),&c(x)=1,\\
  1,&c(x)=2.
 \end{cases}
\]
Both the level sets of $c$ and $\beta$ are $K$-periodic, so
$b'$ is $K$-periodic. If $I_2(y)=\{f_0\}$, then
$\sum_{f\in F}b'(y-f)=b'(y-f_0)=1$. If
$I_1(y)=\{f_1,f_2\}$, the same sum equals
$\beta(y-f_1)+\beta(y-f_2)=1$. Thus $1_F*b'=1$ everywhere,
and $\{x:b'(x)=1\}$ is the required periodic complement.
\end{proof}

\section{Tiles of cardinality twice a prime}\label{sec:twice-prime}

We prove the $2q$ case in three steps. The common zero set first
yields either a lattice complement or a finite spectral remainder.
We use this remainder to construct a periodic twofold covering,
then apply Proposition~\ref{prop:split}. The final subsection
recovers the known four-point endpoint.

\subsection{The spectral filtering dichotomy}\label{sec:filter}

Until Subsection~\ref{sec:four}, let $q$ be an odd prime and let
$r\geq0$. We work in a free abelian group $G\cong\mathbb Z^r$ with
$F\subseteq G$, $|F|=2q$, $0\in F$, and $G=\langle F\rangle$.
In the notation of Section~\ref{sec:algebra}, specialize to $Q=2q$ and write
\begin{equation}\label{eq:common-zero-set}
 P_s(\theta)=\sum_{f\in F}\theta(sf),\qquad
 Z_F=\bigcap_{j=0}^{2q-1}\{\theta\in\widehat G:
                         P_{1+2qj}(\theta)=0\}.
\end{equation}

The set $Z_F$ is compact, but it need not be finite. We record both
assertions, including an explicit example for the second one.

\begin{proposition}[Compactness and possible infinitude of the common zero set]
\label{prop:common-zero-geometry}
For every odd prime $q$, every finite-rank free abelian group $G$,
and every set $F\subseteq G$ with $|F|=2q$, $0\in F$, and
$G=\langle F\rangle$, the set $Z_F$ in \eqref{eq:common-zero-set}
is compact in $\widehat G$.

For every odd prime $q$, these hypotheses can hold while $Z_F$ is
infinite. Explicitly, let $G=\mathbb Z^2$ with standard basis
$e_1,e_2$ and set
\[
 F=\{a e_1+b e_2:a\in\{0,1\},\ 0\leq b<q\}.
\]
Then $|F|=2q$, $0\in F$, $G=\langle F\rangle$, and $Z_F$ is
infinite.
\end{proposition}

\begin{proof}
Fix an odd prime $q$, a finite-rank free abelian group
$G\cong\mathbb Z^r$, and a set $F$ satisfying the hypotheses. Since
$\widehat G\cong\mathbb T^r$, the dual group $\widehat G$ is compact.

For each $g\in G$, the evaluation map
$\theta\mapsto\theta(g)$ is continuous in the product topology on
$\widehat G$. Hence every dilation polynomial
\[
 P_s(\theta)=\sum_{f\in F}\theta(sf)
\]
is continuous. Equation~\eqref{eq:common-zero-set} expresses $Z_F$
as a finite intersection of inverse images
$P_{1+2qj}^{-1}(\{0\})$. These inverse images are closed, so $Z_F$
is a closed subset of the compact space $\widehat G$ and is therefore
compact.

For the infinitude claim, fix an odd prime $q$, take
$G=\mathbb Z^2$ with standard basis $e_1,e_2$, and let
\[
 F=\{a e_1+b e_2:a\in\{0,1\},\ 0\leq b<q\}.
\]
The displayed parametrization is injective, so $|F|=2q$, and it
includes $0$. Since $q\geq3$, it also includes $e_1$ and $e_2$,
whence $G=\langle F\rangle$.

For every $t\in\mathbb C$ with $|t|=1$, define a character
$\theta_t\in\widehat G$ by
\[
 \theta_t(me_1+ne_2)=(-1)^m t^n
 \qquad(m,n\in\mathbb Z).
\]
If $s=1+2qj$ with $0\leq j<2q$, then $s$ is odd and
\[
\begin{aligned}
 P_s(\theta_t)
 &=\sum_{a=0}^1\sum_{b=0}^{q-1}(-1)^{sa}t^{sb}\\
 &=\bigl(1+(-1)^s\bigr)\sum_{b=0}^{q-1}t^{sb}=0.
\end{aligned}
\]
Therefore $\theta_t\in Z_F$ for every $t$ on the unit circle. Since
$\theta_t(e_2)=t$, these characters are pairwise distinct, so $Z_F$
is infinite.
\end{proof}

Proposition~\ref{prop:common-zero-geometry} shows that $Z_F$ may be
infinite. The following proposition controls the subset of $Z_F$ that
survives multiplication by $P_q$.

\begin{proposition}[Filtering the opposite-pair components]\label{prop:filter}
Under the preceding assumptions, either $F$ has a lattice complement
in $G$, or there is a finite set $S\subseteq\widehat G$ of torsion
characters such that
\begin{equation}\label{eq:filter}
 \{\theta\in Z_F:P_q(\theta)\neq0\}\subseteq S.
\end{equation}
The set $S$ may be chosen to depend only on $F$.
\end{proposition}

\begin{proof}
\emph{Step 1: cancellation of opposite pairs.}
Fix $\theta\in Z_F=Z_{F,2q}$ and write $z_f=\theta(f)$.
By Lemma~\ref{lem:blocks}, partitioning $F$ according to the values
$z_f^{2q}$ gives vanishing blocks.

Within each block, divide all phases by one chosen phase in that block.
Every resulting ratio belongs to $\mu_{2q}$. Lemma~\ref{lem:roots}
therefore applies separately to the blocks and then to their collection.
All $2q$ phases can be partitioned into $q$ opposite pairs or into
two rotated $q$-cycles. In the first case, each opposite pair contributes
$z^q+(-z)^q=0$ to $P_q(\theta)$, because $q$ is odd.

\smallskip\noindent\emph{Step 2: the rank-deficient alternative.}
Suppose now that $P_q(\theta)\neq0$. The decomposition must then
consist of two $q$-cycles. It determines a partition
$F=F_0\sqcup F_1$, with $|F_0|=|F_1|=q$. Label the part containing
zero as $F_0$, choose $a_0=0$ and $a_1\in F_1$, and put
\begin{equation}\label{eq:within-cycle-group}
 H=\langle f-a_i:f\in F_i,\ i=0,1\rangle,
 \qquad a=a_1.
\end{equation}
For $f\in F_i$, write $f=a_i+(f-a_i)$. Thus
$G=H+\mathbb Za$. The cycle condition says that
$\{\theta(f-a_i):f\in F_i\}=\mu_q$, with each value occurring
once. Consequently $\theta(H)\subseteq\mu_q$ and
$\operatorname{rank}H\geq r-1$. It also gives the explicit
multiplier value $P_q(\theta)=q(1+\theta(a)^q)$.

If $\operatorname{rank}H=r-1$, put $\zeta=\exp(2\pi i/q)$
and write $\theta(h)=\zeta^{\varphi(h)}$ for $h\in H$.
The homomorphism $\varphi\colon H\to\mathbb Z/q\mathbb Z$ maps
each normalized block bijectively onto $\mathbb Z/q\mathbb Z$.
Moreover, $G/H$ is generated by $a+H$ and has rank one. If a
nonzero multiple of $a$ belonged to $H$, then this cyclic quotient
would be finite, contrary to its rank. Hence $H\cap\mathbb Za=\{0\}$
and $G=H\oplus\mathbb Za$. We may therefore define the quotient map
\begin{equation}\label{eq:lattice-quotient}
 \Psi\colon H\oplus\mathbb Za\longrightarrow
              (\mathbb Z/q\mathbb Z)\times(\mathbb Z/2\mathbb Z),
 \qquad \Psi(h+ma)=(\varphi(h),m\bmod2).
\end{equation}
The points of $F_0$ map bijectively to the layer with second coordinate
zero, and those of $F_1$ map bijectively to the layer with second
coordinate one. Hence $F$ is a complete set of representatives for
$\ker\Psi$, so $F\oplus\ker\Psi=G$.

\smallskip\noindent\emph{Step 3: finite torsion support.}
If $F$ has no lattice complement, this rank-deficient alternative is
impossible. Thus every partition arising from a character with
$P_q(\theta)\neq0$ has $\operatorname{rank}H=r$. Since
$\theta(H)\subseteq\mu_q$, we have $\theta\in(qH)^\perp$.
Using the coordinates in \eqref{eq:annihilator-coordinates}, write
$H=\bigoplus_{i=1}^r m_i\mathbb Ze_i$. Then
$qH=\bigoplus_{i=1}^r qm_i\mathbb Ze_i$, so $(qH)^\perp$ is a
finite torsion group of order
\[
 |(qH)^\perp|=[G:qH]=\prod_{i=1}^r(qm_i)=q^r[G:H].
\]

There are only finitely many partitions of $F$ into two $q$-element
parts. For each partition, the subgroup in
\eqref{eq:within-cycle-group} is independent of the chosen base point
inside either part: it is generated by all differences within each part.
Take $S$ to be the union of $(qH)^\perp$ over the full-rank partitions.
This finite set depends only on $F$ and satisfies \eqref{eq:filter}.
\end{proof}

\subsection{Construction of a periodic twofold covering}\label{sec:cover}

We now realize the filtered spectrum as a twofold covering.
The Frobenius congruence supplies its integer multiplicities, while
Proposition~\ref{prop:filter} supplies periodicity.

\begin{proposition}[Periodic twofold covering]\label{prop:cover}
Let $G\cong\mathbb Z^r$, and let $F\subseteq G$ be a tile of
cardinality $2q$, where $q$ is an odd prime. Suppose that $0\in F$
and $G=\langle F\rangle$. Then either $F$ has a lattice complement,
or there is a tiling indicator $a\colon G\to\{0,1\}$ such that
\begin{equation}\label{eq:two-cover}
 c=\frac1q1_{qF}*a\quad\text{is periodic},\qquad
 c(G)\subseteq\{0,1,2\},\qquad 1_F*c=2.
\end{equation}
Here periodicity is in the sense of Section~\ref{sec:split}: $c$ is
invariant under a finite-index subgroup of $G$.
For this $a$ and any fixed $f_0\in F$, the tiling indicator
$b(x)=a(x-qf_0)$ satisfies
\begin{equation}\label{eq:support-constraints}
 b(x)=0\ \text{when }c(x)=0,\qquad
 b(x)=1\ \text{when }c(x)=2.
\end{equation}
\end{proposition}

\begin{proof}
\emph{Step 1: the covering and support identities.}
We first establish the covering and support identities for every tiling
indicator $a$. In the group ring $(\mathbb Z/q\mathbb Z)[G]$, the
Frobenius identity gives $1_{qF}\equiv1_F^{*q}$. Convolving with $a$
and using $1_F*a=1$ gives
\begin{equation}\label{eq:frobenius-q}
\begin{aligned}
 1_{qF}*a
 &\equiv1_F^{*q}*a
  =1_F^{*(q-1)}*(1_F*a)\\
 &=1_F^{*(q-1)}*1=(2q)^{q-1}1\equiv0\pmod q.
\end{aligned}
\end{equation}
The left-hand side is the sum of $2q$ binary values and lies between
$0$ and $2q$. Therefore its quotient by $q$ belongs to $\{0,1,2\}$.
Associativity gives $1_F*c=q^{-1}1_{qF}*(1_F*a)=2$.
If $c(x)=0$, every term $a(x-qf)$ equals zero; if $c(x)=2$,
every such term equals one. This proves \eqref{eq:support-constraints}.
Translation preserves the tiling equation, so $b$ is a tiling indicator.

\smallskip\noindent\emph{Step 2: periodicity of the covering observable.}
Assume henceforth that $F$ has no lattice complement. It remains to
choose a tiling indicator $a$ for which $c$ is periodic.

Choose an invariant probability measure $\mu$ on $X_F$ as in
Subsection~\ref{subsec:systems}. Ergodicity is unnecessary in this case.
Let $f(a)=a(0)$ and $h=f-1/(2q)$. Lemma~\ref{lem:dilation}
gives, for $0\leq j<2q$,
\begin{equation}\label{eq:annihilation}
 \sum_{v\in F}U_{(1+2qj)v}h=0.
\end{equation}

Equations~\eqref{eq:annihilation} and~\eqref{eq:spectral-norm} imply
$\int|P_{1+2qj}|^2\,d\sigma_h=0$ for every $0\leq j<2q$.
Since the polynomials are continuous, $\sigma_h$ is supported on
$Z_F$. Consider the observables
\begin{equation}\label{eq:cover-observable}
 C_0=\frac1q\sum_{v\in F}U_{qv}h,
 \qquad C=\frac1q\sum_{v\in F}U_{qv}f=\frac1q+C_0.
\end{equation}
The spectral multiplier identity yields
\begin{equation}\label{eq:filtered-measure}
 d\sigma_{C_0}(\theta)=\frac{|P_q(\theta)|^2}{q^2}\,d\sigma_h(\theta).
\end{equation}
Proposition~\ref{prop:filter} supplies a finite torsion set $S$.
Since $\operatorname{supp}\sigma_h\subseteq Z_F$,
\[
 \sigma_{C_0}(\widehat G\setminus S)
 =\frac1{q^2}\int_{Z_F\setminus S}|P_q(\theta)|^2\,
 d\sigma_h(\theta)=0.
\]
The finite-torsion-spectrum consequence \eqref{eq:torsion-periods},
applied to $C_0$, gives the finite-index subgroup
\[
 \Gamma=\bigcap_{\theta\in S}\ker\theta,
\]
with $\Gamma=G$ if $S$ is empty, and gives
$U_\gamma C_0=C_0$ for every $\gamma\in\Gamma$.
Since $C=C_0+1/q$, also $U_\gamma C=C$ in $L^2$.

\smallskip\noindent\emph{Step 3: realization on a tiling configuration.}
For a configuration $a\in X_F$, let $c_a=q^{-1}1_{qF}*a$.
The shift convention gives
$C(T_xa)=q^{-1}\sum_{v\in F}(T_xa)(-qv)
=q^{-1}\sum_{v\in F}a(x-qv)=c_a(x)$.
For fixed $x\in G$ and $\gamma\in\Gamma$, the identity
$U_{-\gamma}C=C$ and invariance of $\mu$ imply
$C(T_{x+\gamma}a)=C(T_xa)$ for almost every $a$.
Consequently the countable intersection
\[
 X_* =\bigcap_{x\in G}\bigcap_{\gamma\in\Gamma}
 \{a\in X_F:C(T_{x+\gamma}a)=C(T_xa)\}
\]
has measure one. Choose $a\in X_*$; then
$c_a(x+\gamma)=c_a(x)$ simultaneously for all $x$ and $\gamma$.
Since $a\in X_F$, its tiling equation holds at every coordinate,
and $c_a$ is pointwise $\Gamma$-periodic.
The covering and support identities already proved for all $a\in X_F$
complete the argument.
\end{proof}

\subsection{Completion of the twice-prime case}\label{subsec:twice-completion}

The preceding construction meets exactly the support requirements
of Proposition~\ref{prop:split}.

\begin{proof}[Proof of Theorem~\ref{thm:main} when $|F|=2q$]
By Lemma~\ref{lem:intrinsic-reduction}, translate $F$ to contain zero
and work in $G=\langle F\rangle$. If $F$ has a lattice complement
$L$ in $G$, then Lemma~\ref{lem:lattice-index} gives
$[G:L]=|F|$. Since $L+L=L$, the subgroup $L$ is invariant under
itself and is therefore a fully periodic tiling complement in $G$.
The conclusion then follows from Lemma~\ref{lem:extend}.
Otherwise Proposition~\ref{prop:cover} gives a periodic twofold
covering and a compatible tiling indicator.
Proposition~\ref{prop:split} gives a fully periodic complement
in $G$, and Lemma~\ref{lem:extend} extends it to $\mathbb Z^d$.
Undoing the initial translation completes the proof.
\end{proof}

\subsection{The four-point endpoint}\label{sec:four}

The cancellation $z^q+(-z)^q=0$ in the filtering argument uses
that $q$ is odd. At $q=2$, the block decomposition itself forces
finite torsion spectrum, unless a lattice complement already exists.
This gives the following proof of Szegedy's four-point theorem
\cite{Szegedy}.

\begin{proposition}[Szegedy's four-point case]\label{prop:four}
For every integer $d\geq1$, every four-point tile
$F\subseteq\mathbb Z^d$ admits a fully periodic tiling complement.
\end{proposition}

\begin{proof}
\emph{Step 1: opposite pairs.}
Fix an integer $d\geq1$ and a four-point tile
$F\subseteq\mathbb Z^d$. By Lemma~\ref{lem:intrinsic-reduction},
translate the tile and work first in
$G=\langle F\rangle\cong\mathbb Z^r$ with $0\in F$.
Using the dilation polynomials from Section~\ref{sec:algebra}, set
\[
 Z_F^{(4)}=\bigcap_{j=0}^{3}\{\theta:P_{1+2j}(\theta)=0\}.
\]
For $\theta\in Z_F^{(4)}$, partition the four phases
$z_f=\theta(f)$ by their squares. The Vandermonde argument in
Lemma~\ref{lem:blocks} shows that each block has zero sum. The phases
in a block are all $z$ or $-z$ for a fixed $z$, so a vanishing block
consists of equally many copies of these two values. Thus all four
phases split into two opposite pairs $F_0,F_1$.

\smallskip\noindent\emph{Step 2: a lattice complement or finite spectrum.}
Write $F_0=\{0,d_0\}$ and $F_1=\{a_1,a_1+d_1\}$.
The opposite-pair condition gives $\theta(d_0)=\theta(d_1)=-1$.
For $H=\langle d_0,d_1\rangle$, we therefore have
$G=H+\mathbb Za_1$, $\theta(H)\subseteq\mu_2$, and
$\operatorname{rank}H\geq r-1$. If the rank is $r-1$, define
$\chi\colon H\to\mathbb Z/2\mathbb Z$ by
$\theta(h)=(-1)^{\chi(h)}$. This homomorphism is bijective on each
normalized pair. Lemma~\ref{lem:quotient}, with $b=2$ and
$K=\mathbb Z/2\mathbb Z$, constructs the quotient map
\[
 \Psi\colon G=H\oplus\mathbb Za_1\longrightarrow
 (\mathbb Z/2\mathbb Z)^2,
 \qquad \Psi(h+ma_1)=(\chi(h),m\bmod2),
\]
and gives $F\oplus\ker\Psi=G$. Thus $L=\ker\Psi$ is a lattice
complement. Its index is $[G:L]=|F|=4$ by
Lemma~\ref{lem:lattice-index}. Since $L+L=L$, it is invariant under
itself and is a fully periodic tiling complement in $G$.
Otherwise $H$, and hence $2H$, has full rank; moreover
$\theta\in(2H)^\perp$. The annihilator-coordinate description
\eqref{eq:annihilator-coordinates} shows that $(2H)^\perp$ is a
finite torsion group. If there is no lattice complement, taking the
union over the finitely many pair partitions shows that the entire
set $Z_F^{(4)}$ is contained in a finite torsion set.

\smallskip\noindent\emph{Step 3: a periodic tiling.}
In the remaining case, choose an invariant probability measure on
$X_F$ and let $f(a)=a(0)$. Here $Z_F^{(4)}=Z_{F,2}$, so
Lemma~\ref{lem:common-support}, equivalently the support inclusion
\eqref{eq:general-support} with $n=4$ and $Q=2$, shows that
$\sigma_f$ is supported on the finite torsion set
$Z_F^{(4)}\cup\{\mathbf1\}$. The finite-torsion-spectrum consequence
\eqref{eq:torsion-periods} gives a finite-index subgroup
$\Gamma\leq G$ such that $U_\gamma f=f$ for every
$\gamma\in\Gamma$. As in Step~3 of the proof of
Proposition~\ref{prop:cover}, intersect the conull sets on which
$f(T_{x+\gamma}a)=f(T_xa)$ over all $x\in G$ and
$\gamma\in\Gamma$. Every configuration in this intersection
is a pointwise $\Gamma$-periodic tiling indicator.
In both alternatives, Lemma~\ref{lem:extend} extends the complement
to $\mathbb Z^d$; undoing the initial translation gives the result
for the original tile.
\end{proof}

\section{Periodicity from rational affine circles}\label{sec:circles}

For nine-point tiles, the algebraic reduction may leave spectral
circles instead of a finite torsion set. We now prove
Theorem~\ref{thm:circle} to handle this alternative. We first
obtain a conditional-density dichotomy for a general ergodic lattice
action. We then use the tiling partition to replace the components
of density $1/2$ by periodic two-colorings.

\subsection{Spectral decomposition and conditional laws}\label{subsec:statistics}

Let $G=\mathbb Z^d$, $d\geq2$, act ergodically by
probability-preserving transformations on a standard probability
space $(X,\mathcal B,\mu)$. For $g\in G$, $L\leq G$, and
$h\in L^2(X,\mu)$, write
\[
 U_gh=h\circ T_{-g},\qquad P_Lh=\mathbb E(h\mid\mathcal I_L),
\]
and let $\sigma_h$ be the spectral measure of $h$ for $U$.
Proposition~\ref{prop:invariant-projection} and
Lemma~\ref{lem:invariant-bounds} apply directly to this action. If
$\mathcal H_v=\overline{\operatorname{span}}\{U_gv:g\in G\}$, there
is a cyclic spectral multiplier model
\[
 W_v\colon\mathcal H_v\longrightarrow L^2(\widehat G,\sigma_v),
 \qquad W_vv=1,
\]
such that, for $w\in\mathcal H_v$,
\[
 d\sigma_w=|W_vw|^2\,d\sigma_v,
 \qquad
 W_v(P_Lw)=1_{L^\perp}W_vw.
\]
Also, positivity and unitality of conditional expectation give
$0\leq P_Lh\leq1$ whenever $h$ is real-valued and $0\leq h\leq1$.

Let $D\in\mathcal B$, put $f=1_D$, and suppose that $\sigma_f$ is
supported on finitely many rational affine circles. We first record
the coordinate description of the direction of a rational circle.

\begin{lemma}[One-dimensional subtori]\label{lem:one-dimensional-subtori}
Let $\mathbb T=\mathbb R/\mathbb Z$, and call
$v=(v_1,\ldots,v_d)\ne0$ \emph{primitive} if
$v_1\mathbb Z+\cdots+v_d\mathbb Z=\mathbb Z$. Here a
one-dimensional subtorus of $\mathbb T^d$ means a subgroup $H$,
with the subspace topology, for which there is a topological-group
isomorphism $\mathbb T\to H$. Every such subtorus has the form
\[
 H=\mathbb T v:=\{tv:t\in\mathbb T\}
\]
for a primitive vector $v\in\mathbb Z^d$, and this vector is unique
up to sign. Consequently the rational line $\mathbb Qv\subseteq
\mathbb Q^d$ depends only on $H$.
\end{lemma}
\begin{proof}
Write $p\colon\mathbb R\to\mathbb T$ for the quotient map. We first show
that every continuous homomorphism $\psi\colon\mathbb T\to\mathbb T$ has
the form $\psi(t)=nt$ for a unique $n\in\mathbb Z$. By the
covering-space lifting criterion~\cite[Proposition~1.33]{Hatcher}, the
map $\psi\circ p$ has a continuous lift
$\widetilde\psi\colon\mathbb R\to\mathbb R$ with $\widetilde\psi(0)=0$.
For $x,y\in\mathbb R$,
\[
 \widetilde\psi(x+y)-\widetilde\psi(x)-\widetilde\psi(y)\in\mathbb Z.
\]
This is a continuous integer-valued function of $(x,y)$ and is zero
at $(0,0)$, so it is identically zero. Thus $\widetilde\psi$ is
additive, and continuity gives $\widetilde\psi(x)=cx$ for
$c=\widetilde\psi(1)$. Since $p(x+1)=p(x)$, one has $p(c)=0$, hence
$c\in\mathbb Z$. This also proves uniqueness.

Let $H\leq\mathbb T^d$ be a one-dimensional subtorus, and choose a
topological-group isomorphism $\rho\colon\mathbb T\to H$. Composing with
the coordinate projections gives continuous homomorphisms
$\rho_j\colon\mathbb T\to\mathbb T$, so there are integers $v_j$ such
that $\rho_j(t)=v_jt$. For $v=(v_1,\ldots,v_d)$, therefore,
\[
 \rho(t)=tv,\qquad H=\mathbb T v.
\]
The subgroup $v_1\mathbb Z+\cdots+v_d\mathbb Z$ is $m\mathbb Z$ for
a unique $m\geq0$. If $m=0$, then $v=0$ and $\rho$ is not injective;
if $m>1$, then $1/m\in\mathbb T$ is a nonzero element killed by
$t\mapsto tv$. Hence injectivity forces $m=1$, which is precisely
primitivity. Conversely, if $v$ is primitive, choose integers $a_j$
with $\sum_{j=1}^da_jv_j=1$. If $tv=0$ in $\mathbb T^d$, then
$t=\sum_{j=1}^da_j(v_jt)=0$ in $\mathbb T$, so $t\mapsto tv$ is
injective.

Finally, suppose $\mathbb T v=\mathbb T w=:H$ for primitive $v,w$.
The map
\[
 (t\mapsto tv)^{-1}\circ(t\mapsto tw)\colon\mathbb T\longrightarrow\mathbb T
\]
is an automorphism, hence multiplication by an integer $n$. Its
inverse is also multiplication by an integer, so $n=\pm1$.
Coordinatewise uniqueness then gives $w=nv=\pm v$.
\end{proof}

Choose an integer $N\geq1$ and rational affine circles
$C_1,\ldots,C_N\subseteq\widehat G$ such that
\[
 \operatorname{supp}\sigma_f\subseteq S:=\bigcup_{k=1}^N C_k.
\]
If the spectral support is empty, take $N=1$ and choose an arbitrary
rational affine circle.
Use the standard-coordinate isomorphism
\[
 \Phi\colon\mathbb T^d\longrightarrow\widehat G,
 \qquad
 \Phi(x)(g)=\exp\bigl(2\pi i\langle g,x\rangle_{\mathbb T}\bigr),
 \qquad
 \langle g,x\rangle_{\mathbb T}=\sum_{r=1}^d g_rx_r\in\mathbb T.
\]
For $g,v\in\mathbb Z^d$, write
$g\mathbin{\cdot}v=\sum_{r=1}^dg_rv_r\in\mathbb Z$. By
Lemma~\ref{lem:one-dimensional-subtori}, for every $k$ we may choose
a torsion point $a_k\in\mathbb T^d$ and a primitive vector
$w_k\in\mathbb Z^d\setminus\{0\}$ such that
\[
 C_k=\Phi(a_k+\mathbb T w_k).
\]
List the distinct directions among the $\mathbb Qw_k$ as
$\mathbb Qv_1,\ldots,\mathbb Qv_s$, choosing the $v_i$ primitive.
Thus $s\geq1$ and the $v_i$ are pairwise nonparallel. Set
\[
 S_i=\{a_k:\mathbb Qw_k=\mathbb Qv_i\},
 \qquad
 S_i+\mathbb T v_i=\{a+tv_i:a\in S_i,\ t\in\mathbb T\}.
\]
Then each $S_i$ is a nonempty finite set of torsion points and
\[
 S=\bigcup_{i=1}^s\Phi(S_i+\mathbb T v_i).
\]
Henceforth, $i,j\in\{1,\ldots,s\}$. Choose $q_i\geq1$ such that
$q_i a=0$ in $\mathbb T^d$ for every
$a\in S_i$, and define
\[
 H_i=\{g\in G:g\mathbin{\cdot}v_i=0\},
 \qquad
 \Lambda_i=q_iH_i,
 \qquad
 f_i=P_{\Lambda_i}f,
 \qquad
 f_0=f-\sum_{i=1}^s f_i.
\]

\begin{lemma}[Rational-circle data]\label{lem:rational-circle-data}
For the data above,
\[
 \Phi(S_i+\mathbb T v_i)\subseteq\Lambda_i^\perp,
 \qquad \operatorname{rank}\Lambda_i=d-1,
\]
and
\[
 \operatorname{rank}(\Lambda_i+\Lambda_j)=d\qquad(i\ne j).
\]
Moreover,
\[
 0\leq f_i\leq1\qquad\text{$\mu$-almost everywhere}.
\]
\end{lemma}
\begin{proof}
Fix $i$, take $a\in S_i$ and $t\in\mathbb T$, and let
$\lambda=q_ih\in\Lambda_i$ with $h\in H_i$. Then
\[
\begin{aligned}
 \Phi(a+tv_i)(\lambda)
 &=\exp\!\left(2\pi i q_i
   \left(\sum_{r=1}^d h_ra_r+t(h\mathbin{\cdot}v_i)\right)\right)\\
 &=1.
\end{aligned}
\]
Indeed, $q_ia=0$ makes the first term zero in $\mathbb T$, while
$h\mathbin{\cdot}v_i=0$ makes the second term zero. This proves the
annihilator inclusion.

Because $v_i$ is primitive, B\'ezout's identity gives $u_i\in G$ with
$u_i\mathbin{\cdot}v_i=1$. Every $g\in G$ has the unique
decomposition
\[
 g=\bigl(g-(g\mathbin{\cdot}v_i)u_i\bigr)
 +(g\mathbin{\cdot}v_i)u_i\in H_i\oplus\mathbb Zu_i.
\]
Thus $H_i$, and hence $\Lambda_i=q_iH_i$, has rank $d-1$. If
$i\ne j$, the rational spans of $H_i$ and $H_j$ are distinct
hyperplanes, since their normal vectors $v_i$ and $v_j$ are
nonparallel. Their sum is therefore $\mathbb Q^d$, so
$\Lambda_i+\Lambda_j$ has rank $d$. Finally, positivity and unitality
of conditional expectation give
$0\leq P_{\Lambda_i}f=f_i\leq1$ almost everywhere.
\end{proof}

\begin{lemma}[Periodic correction]\label{lem:periodic-correction}
There is a finite-index subgroup $\Gamma_0\leq G$ such that
\[
 U_\gamma f_0=f_0\quad\text{in }L^2(X,\mu)
 \qquad(\gamma\in\Gamma_0).
\]
For every $i\in\{1,\ldots,s\}$ and every integer $a\geq1$, there is
a finite-index subgroup $\Gamma_{i,a}\leq G$ such that
\begin{equation}\label{eq:periodic-correction}
 U_\gamma\bigl(P_{a\Lambda_i}f-f_i\bigr)
 =P_{a\Lambda_i}f-f_i\quad\text{in }L^2(X,\mu)
 \qquad(\gamma\in\Gamma_{i,a}).
\end{equation}
\end{lemma}
\begin{proof}
Write $A_i=\Lambda_i^\perp$. In the cyclic spectral multiplier model
$W_f$ recalled above,
\[
 W_ff_i=1_{A_i},
 \qquad
 W_ff_0=m_0:=1-\sum_{i=1}^s1_{A_i},
 \qquad
 d\sigma_{f_0}=|m_0|^2\,d\sigma_f.
\]
The support assumption and Lemma~\ref{lem:rational-circle-data} show
that $\sigma_f$ is supported on $\bigcup_{i=1}^sA_i$. On this union,
$m_0$ vanishes unless a point belongs to at least two of the $A_i$.
Since
\[
 A_i\cap A_j=(\Lambda_i+\Lambda_j)^\perp,
\]
it follows that
\[
 \operatorname{supp}\sigma_{f_0}
 \subseteq E_0:=
 \bigcup_{1\leq i<j\leq s}(\Lambda_i+\Lambda_j)^\perp.
\]
Each $\Lambda_i+\Lambda_j$ has full rank by
Lemma~\ref{lem:rational-circle-data}, so its annihilator is a finite
set of torsion characters by the coordinate description
\eqref{eq:annihilator-coordinates}. Thus $E_0$ is a finite torsion
set.

We record the resulting kernel argument. If a vector $v$ in the
cyclic space of $f$ has multiplier supported on a finite set $E$ of
torsion characters, put
\[
 \Gamma_E=\bigcap_{\theta\in E}\ker\theta,
\]
with $\Gamma_E=G$ when $E$ is empty. A common multiple $n$ of the
orders of the characters in $E$ satisfies $nG\leq\Gamma_E$, so
$\Gamma_E$ has finite index. For $\gamma\in\Gamma_E$, the
intertwining identity gives
\[
 W_f(U_\gamma v)(\theta)
 =\theta(\gamma)W_fv(\theta)=W_fv(\theta)
\]
for $\sigma_f$-almost every $\theta$. Hence $U_\gamma v=v$ in
$L^2(X,\mu)$. Applying this to $f_0$ and $E_0$ gives $\Gamma_0$.

Fix $i$ and $a\geq1$, put $B=(a\Lambda_i)^\perp$, and set
$w=P_{a\Lambda_i}f-f_i$. Since $a\Lambda_i\leq\Lambda_i$, one has
$A_i\subseteq B$. The invariant-projection multiplier formula gives
\[
 W_fw=1_B-1_{A_i}=1_{B\setminus A_i},
 \qquad
 d\sigma_w=1_{B\setminus A_i}\,d\sigma_f.
\]
Because $\sigma_f$ is supported on $\bigcup_{j=1}^sA_j$,
\[
 \operatorname{supp}\sigma_w
 \subseteq E_{i,a}:=
 \bigcup_{\substack{1\leq j\leq s\\j\ne i}}(B\cap A_j)
 =\bigcup_{\substack{1\leq j\leq s\\j\ne i}}
   (a\Lambda_i+\Lambda_j)^\perp.
\]
For $j\ne i$, the group $a\Lambda_i+\Lambda_j$ contains
$a(\Lambda_i+\Lambda_j)$ and therefore has full rank. Hence
$E_{i,a}$ is a finite set of torsion characters. The same kernel
argument, applied to $w$ and $E_{i,a}$, gives $\Gamma_{i,a}$ and
proves \eqref{eq:periodic-correction}.
\end{proof}

Fix a finite-index subgroup $\Gamma_0$ as in
Lemma~\ref{lem:periodic-correction}. We next determine the laws of
the summands after restriction to finite-index ergodic components.
A bar denotes reduction modulo $\mathbb Z$.

\begin{lemma}[Stable conditional laws]\label{lem:conditional-laws}
For every $j\in\{1,\ldots,s\}$, there is a finite-index subgroup
$\Gamma_j\leq G$ such that, for every $H\leq\Gamma_j$ with
$[\Gamma_j:H]<\infty$ and every $H$-ergodic component $E$,
\begin{equation}\label{eq:hereditary}
 (\overline f_j)_*\mu_E
 \quad\text{is Haar probability measure on $\mathbb R/\mathbb Z$,
 or $\delta_z$ for some $z\in\mathbb R/\mathbb Z$.}
\end{equation}
\end{lemma}
\begin{proof}
\emph{Step 1: a polynomial formula along transverse orbits.}
Fix $j$. Choose $h_0\in G$ outside every
$\operatorname{span}_{\mathbb Q}\Lambda_i$; this is possible because
a finite union of proper rational subspaces does not cover
$\mathbb Q^d$. For each $i\ne j$, the full-rank assertion in
Lemma~\ref{lem:rational-circle-data} shows that $\Lambda_i$ is not
contained in $\operatorname{span}_{\mathbb Q}\Lambda_j$, so choose
\[
 w_i\in\Lambda_i\setminus\operatorname{span}_{\mathbb Q}\Lambda_j.
\]
The sum $L_j=\Lambda_j\oplus\mathbb Zh_0$ is direct and has full
rank, hence finite index by Smith normal form. Therefore
$\Gamma_0\cap L_j$ has finite index, and the finite quotient
$G/(\Gamma_0\cap L_j)$ has finite exponent. Choose $a\geq1$ such that
\[
 aG\subseteq\Gamma_0\cap L_j.
\]
For circle-valued maps use the same notation $U_gv=v\circ T_{-g}$,
with addition and subtraction in $\mathbb R/\mathbb Z$. Since $f$ is
integer-valued and $f=f_0+\sum_{i=1}^sf_i$, reduction modulo one gives
\[
 \overline f_j=-\overline f_0-
 \sum_{\substack{1\leq i\leq s\\i\ne j}}\overline f_i.
\]
The first factor in the following operator kills $\overline f_0$,
and the factor indexed by $i$ kills $\overline f_i$:
\begin{equation}\label{eq:factor-relation}
 (U_{ah_0}-I)
 \prod_{\substack{1\leq i\leq s\\i\ne j}}
 (U_{aw_i}-I)\overline f_j=0.
\end{equation}
All the factors commute.

For $i\ne j$, write uniquely
$aw_i=b_ih_0+\lambda_i$, with $b_i\in\mathbb Z$ and
$\lambda_i\in\Lambda_j$. Here $b_i\ne0$; replacing $w_i$ by $-w_i$
if needed makes $b_i>0$. This replacement preserves
\eqref{eq:factor-relation}, because the new factor still annihilates
$\overline f_i$. Since $\overline f_j$ is $\Lambda_j$-invariant,
$U_{aw_i}\overline f_j=U_{b_ih_0}\overline f_j$. Let $m$ be a
positive common multiple of $a$ and all the $b_i$, with the latter
list empty when $s=1$. In $\mathbb Z[t]$,
\[
\begin{split}
 (t^m-1)^s
 ={}&(t^a-1)
 \prod_{\substack{1\leq i\leq s\\i\ne j}}(t^{b_i}-1)
 \left(\sum_{\ell=0}^{m/a-1}t^{\ell a}\right)
 \prod_{\substack{1\leq i\leq s\\i\ne j}}
 \left(\sum_{\ell=0}^{m/b_i-1}t^{\ell b_i}\right).
\end{split}
\]
Substituting $U_{h_0}$ and using \eqref{eq:factor-relation}, then
putting $h_j=mh_0$, gives
\begin{equation}\label{eq:polynomial-relation}
 (U_{h_j}-I)^s\overline f_j=0,
 \qquad
 U_\lambda\overline f_j=\overline f_j\quad(\lambda\in\Lambda_j).
\end{equation}
These identities initially hold almost everywhere. Since $G$ and
$\Lambda_j$ are countable, remove their exceptional sets and all
$G$-translates. This gives a $G$-invariant conull set $X_0$ on which
every translated instance of \eqref{eq:polynomial-relation} holds
pointwise.

For $x\in X_0$, $r\in G$, and $t\in\mathbb Z$, let
$u_{x,r}(t)=\overline f_j(T_{r+th_j}x)$. Because
$U_{h_j}-I$ is the backward difference, evaluating its $s$th power
at $T_{r+(t+s)h_j}x$ gives $(-1)^s\Delta^su_{x,r}(t)$, where
$\Delta u(t)=u(t+1)-u(t)$. Thus $\Delta^su_{x,r}=0$ on $\mathbb Z$.
The Newton formula in $\mathbb R/\mathbb Z$ is
\begin{equation}\label{eq:polynomial-sequence}
\begin{split}
 u_{x,r}(t)&=\sum_{k=0}^{s-1}\binom{t}{k}\alpha_{x,r,k}
                   \pmod{\mathbb Z},\qquad t\in\mathbb Z,\\
 \alpha_{x,r,k}&=\sum_{\ell=0}^k(-1)^{k-\ell}
                        \binom{k}{\ell}u_{x,r}(\ell).
\end{split}
\end{equation}
Both sides of the first identity have zero $s$th difference and agree
at $t=0,\ldots,s-1$, so the recurrence proves the formula in both
integer directions. The group
$\Lambda_j\oplus\mathbb Zh_j$ has finite index. Fix a finite set
$R_j$ of representatives for its cosets in $G$. If
$r=r_0+\lambda+ch_j$, with $r_0\in R_j$, $\lambda\in\Lambda_j$, and
$c\in\mathbb Z$, then
$u_{x,r}(t)=u_{x,r_0}(t+c)$. Thus the finitely many sequences indexed
by $R_j$ control all $r\in G$.

\smallskip\noindent\emph{Step 2: orbit averages on all finite-index refinements.}
Fix $x\in X_0$ and $r_0\in R_j$, choose real lifts
$\widetilde\alpha_{x,r_0,k}$ of the coefficients in
\eqref{eq:polynomial-sequence}, and put
\[
 p_{x,r_0}(t)=\sum_{k=0}^{s-1}\binom{t}{k}
                 \widetilde\alpha_{x,r_0,k}\in\mathbb R[t].
\]
If all its nonconstant binomial coefficients are rational, choose a
common denominator $M$. Every positive $Q$ divisible by $M(s-1)!$
satisfies
\[
 \binom{t+Q}{k}-\binom{t}{k}
 =\sum_{r=1}^k\binom{Q}{r}\binom{t}{k-r}\equiv0\pmod M
 \qquad(1\leq k<s),
\]
so $p_{x,r_0}(t)$ modulo one is periodic. Otherwise the invertible
rational change between the binomial and power bases shows that
$p_{x,r_0}$ has an irrational nonconstant power-basis coefficient,
and Weyl's theorem~\cite{Weyl} makes its reduction modulo one Haar
equidistributed. The same alternative holds on every nonconstant
arithmetic subsequence: if the degree-$k$ coefficient is the highest
irrational nonconstant coefficient, then under $t\mapsto bt+c$, with
$b,c\in\mathbb Z$ and $b\ne0$, the new degree-$k$ coefficient is
$b^k$ times that irrational number plus rational contributions from
higher coefficients. Negative $b$ gives the same conclusion for
two-sided averages.

Choose $q_x\geq1$ divisible by the periods of all periodic sequences
among the finitely many $u_{x,r_0}$, $r_0\in R_j$, taking $q_x=1$ if
there are none, and set
\[
 H_x=\Lambda_j\oplus\mathbb Z(q_xh_j).
\]
Let $H\leq H_x$ satisfy $[H_x:H]<\infty$, and define
\[
 \pi\colon H_x\longrightarrow\mathbb Z,
 \qquad \pi(\lambda+nq_xh_j)=n.
\]
Then $K=H\cap\Lambda_j$ has finite index in $\Lambda_j$, while
$\pi(H)=\ell\mathbb Z$ for a unique $\ell\geq1$. Choose
$b_d\in H$ with $\pi(b_d)=\ell$ and write
$b_d=\lambda+\ell q_xh_j$, $\lambda\in\Lambda_j$. Every element of
$H$ differs from a unique multiple of $b_d$ by an element of $K$, so
\[
 H=K\oplus\mathbb Zb_d.
\]
Choose a basis $b_1,\ldots,b_{d-1}$ of $K$. Then
$b_1,\ldots,b_d$ is a basis of $H$. For every $r\in G$ and every
continuous $\varphi\colon\mathbb R/\mathbb Z\to\mathbb C$,
$\Lambda_j$-invariance gives
\begin{equation}\label{eq:adapted-box-average}
\begin{split}
&\frac{1}{(2N+1)^d}\sum_{n\in\{-N,\ldots,N\}^d}
 \varphi\!\left(\overline f_j\!\left(
 T_{r+\sum_{k=1}^dn_kb_k}x\right)\right)\\
&\hspace{35mm}=
 \frac{1}{2N+1}\sum_{t=-N}^N
 \varphi\bigl(\overline f_j(T_{r+t\ell q_xh_j}x)\bigr).
\end{split}
\end{equation}
Writing $r=r_0+\lambda_0+t_0h_j$ reduces the sequence on the right to
$u_{x,r_0}(t_0+t\ell q_x)$. In the periodic case the choice of $q_x$
makes this sequence constant; otherwise it is Haar equidistributed.
Thus the empirical measures in \eqref{eq:adapted-box-average}
converge weakly to a Dirac mass or Haar measure.

We next identify these orbit limits with conditional laws on one
common conull set. Fix a finite-index subgroup $H\leq G$, an ordered
basis $b=(b_1,\ldots,b_d)$ of $H$, and a bounded measurable function
$\psi$. By Lemma~\ref{lem:components}, modulo null sets $X$ is the
disjoint union of finitely many positive-measure $H$-ergodic
components $E'$. On each such $E'$, the normalized measure
$\mu_{E'}$ is an $H$-invariant ergodic probability measure.

We construct the compact metric system needed for the pointwise
theorem rather than applying that theorem directly to the original
measurable system. Choose $C\geq\lVert\psi\rVert_\infty$ and a
measurable representative of $\psi$ with values in the closed
complex disk $D_C$. Since $H$ is countable, $D_C^H$ is compact and
metrizable. Define the measurable orbit-name map
\[
 \iota_\psi\colon X\longrightarrow D_C^H,
 \qquad \iota_\psi(y)(h)=\psi(T_hy),
\]
and let $K_\psi$ be the closure of $\iota_\psi(X)$. For $g\in H$,
define the continuous shift
\[
 (S_gz)(h)=z(h+g).
\]
Then $S_g\iota_\psi(y)=\iota_\psi(T_gy)$. Hence $K_\psi$ is
$H$-invariant, the shifts give a continuous $H$-action on the compact
metric space $K_\psi$, and
\[
 \nu_{E'}=(\iota_\psi)_*\mu_{E'}
\]
is an invariant Borel probability measure. It is ergodic: the
preimage under $\iota_\psi$ of any invariant Borel set is invariant
modulo $\mu_{E'}$.

Let $e_0\colon K_\psi\to\mathbb C$ be the continuous coordinate map
$e_0(z)=z(0)$. The boxes
\[
 \mathcal F_N(b)=\left\{\sum_{k=1}^dn_kb_k:
                  n_k\in\{-N,\ldots,N\}\right\}
\]
form a tempered F\o lner sequence in $H$: translation by a fixed
element changes only $O(N^{d-1})$ points, while
\[
 \bigcup_{M<N}\bigl(-\mathcal F_M(b)+\mathcal F_N(b)\bigr)
 \subseteq\mathcal F_{2N}(b),
 \qquad
 |\mathcal F_{2N}(b)|\leq2^d|\mathcal F_N(b)|.
\]
The pointwise ergodic theorem~\cite{Lindenstrauss}, applied to
$(K_\psi,\nu_{E'})$ and $e_0$, and then pulled back by $\iota_\psi$,
gives, for $\mu_{E'}$-almost every $y\in E'$,
\begin{equation}\label{eq:compact-factor-average}
 \lim_{N\to\infty}\frac{1}{(2N+1)^d}
 \sum_{n\in\{-N,\ldots,N\}^d}
 \psi\!\left(T_{\sum_{k=1}^dn_kb_k}y\right)
 =\int_{E'}\psi\,d\mu_{E'}.
\end{equation}
On the other hand, $P_H\psi$ is $H$-invariant and therefore constant
almost everywhere on each $H$-ergodic component $E'$. The defining
identity for conditional expectation shows that this constant is
$\int_{E'}\psi\,d\mu_{E'}$.

Take $\psi=\varphi\circ\overline f_j$ for $\varphi$ in a fixed
countable uniformly dense subset of
$C(\mathbb R/\mathbb Z)$. There are only countably many finite-index
subgroups of $G$, their ordered bases, and elements $r\in G$.
Intersect the conull sets from \eqref{eq:compact-factor-average} over
all these choices and their $G$-translates, and also intersect with
$X_0$. We obtain a $G$-invariant conull set $X_*\subseteq X_0$ such
that, for every $x\in X_*$, every such $H,b,r$, and every continuous
$\varphi$,
\begin{equation}\label{eq:conditional-orbit-law}
\begin{split}
&\lim_{N\to\infty}\frac{1}{(2N+1)^d}
 \sum_{n\in\{-N,\ldots,N\}^d}
 \varphi\!\left(\overline f_j\!\left(
 T_{r+\sum_{k=1}^dn_kb_k}x\right)\right)\\
&\hspace{35mm}=P_H(\varphi\circ\overline f_j)(T_rx).
\end{split}
\end{equation}
Uniform approximation extends the identity from the chosen dense
family to every continuous $\varphi$. In particular, the limiting
distribution does not depend on the chosen basis of $H$.

\smallskip\noindent\emph{Step 3: a single subgroup valid almost everywhere.}
Fix an ordered basis $b(H')$ for each finite-index subgroup
$H'\leq G$, and write
\[
 \eta^{H',r}_{N,x}
 =\frac{1}{(2N+1)^d}\sum_{n\in\{-N,\ldots,N\}^d}
 \delta_{\overline f_j(T_{r+\sum_{k=1}^dn_kb_k(H')}x)}.
\]
Let $m_{\mathbb T}$ be Haar probability measure on
$\mathbb T=\mathbb R/\mathbb Z$, and put
\[
 \mathcal D=\{m_{\mathbb T}\}\cup\{\delta_z:z\in\mathbb T\}.
\]
The map $z\mapsto\delta_z$ is continuous for the weak topology.
Since $\mathbb T$ is compact, $\mathcal D$ is compact and hence
closed in the space of probability measures on $\mathbb T$.

For a finite-index subgroup $H\leq G$, let $Y_H\subseteq X_*$ be the
set of $x$ such that, for every $H'\leq H$ with $[H:H']<\infty$ and
every $r\in G$, the measures $\eta^{H',r}_{N,x}$ converge weakly to
a member of $\mathcal D$. The set $Y_H$ is measurable: in a
compatible metric on the probability-measure space, the condition
says that each sequence is Cauchy and its distance from the closed
set $\mathcal D$ tends to zero, and there are only countably many
pairs $(H',r)$. It is $G$-invariant because
\[
 \eta^{H',r}_{N,T_gx}=\eta^{H',r+g}_{N,x}.
\]

For every $x\in X_*$, Step~2 supplies the finite-index subgroup
$H_x$. If $H'\leq H_x$ has finite index, the adapted basis in
\eqref{eq:adapted-box-average} gives a limit in $\mathcal D$, while
\eqref{eq:conditional-orbit-law} shows that the fixed-basis averages
defining $\eta^{H',r}_{N,x}$ have the same limit. Hence $x\in Y_{H_x}$
and
\[
 X_*\subseteq\bigcup_{[G:H]<\infty}Y_H.
\]
There are countably many finite-index subgroups, so some $Y_H$ has
positive measure. Ergodicity and $G$-invariance make it conull; call
this subgroup $\Gamma_j$.

Let $H'\leq\Gamma_j$ satisfy $[\Gamma_j:H']<\infty$, and let $E$ be
an $H'$-ergodic component. By Lemma~\ref{lem:components}, $E$ has
positive measure. For almost every $x\in E\cap Y_{\Gamma_j}$, the
measures $\eta^{H',0}_{N,x}$ converge to some $\nu_x\in\mathcal D$.
For every continuous $\varphi\colon\mathbb T\to\mathbb C$,
\eqref{eq:conditional-orbit-law} and $H'$-ergodicity give
\[
\begin{split}
 \int_{\mathbb T}\varphi\,d\nu_x
 &=P_{H'}(\varphi\circ\overline f_j)(x)\\
 &=\int_E\varphi(\overline f_j(y))\,d\mu_E(y)\\
 &=\int_{\mathbb T}\varphi\,d\bigl((\overline f_j)_*\mu_E\bigr).
\end{split}
\]
Thus $\nu_x=(\overline f_j)_*\mu_E$, which belongs to $\mathcal D$.
This proves \eqref{eq:hereditary}; the index $j$ was arbitrary.
\end{proof}

\subsection{The conditional-density dichotomy}

We now use the bounds $0\leq f_i\leq1$. A Haar law will force
conditional density $1/2$; Dirac laws will yield a period group.
The argument extends \cite[Theorem~4.4]{Bhattacharya} and
\cite[Lemma~C.15]{Khetan} to arbitrary dimension.

\begin{theorem}[Conditional-density dichotomy]\label{thm:dichotomy}
Let $d\geq2$, let $G=\mathbb Z^d$ act ergodically by
probability-preserving transformations on a standard probability
space $(X,\mu)$, let $D\subseteq X$ be measurable, and put $f=1_D$.
Suppose $\sigma_f$ is supported on finitely many rational affine
circles. Then there is a finite-index subgroup $\Gamma\leq G$ such
that, on each $\Gamma$-ergodic component $E$, either
$\mu_E(D)=1/2$ or $D\cap E$ is $(d-1)$-periodic, or both.
\end{theorem}
\begin{proof}
Take $\Gamma=\Gamma_0\cap\bigcap_{j=1}^s\Gamma_j$, with the groups from
Lemmas~\ref{lem:periodic-correction} and
\ref{lem:conditional-laws}, and fix a $\Gamma$-ergodic component
$E$. Lemma~\ref{lem:conditional-laws} gives the Dirac-or-Haar
alternative on every ergodic component of every finite-index
subgroup of $\Gamma$; this stability will be used below.

\emph{Case 1: some conditional law is Haar.}
Suppose $(\overline f_i)_*\mu_E$ is Haar probability measure on
$\mathbb R/\mathbb Z$. By Lemma~\ref{lem:rational-circle-data},
$0\leq f_i\leq1$, and therefore $(f_i)_*\mu_E$ is uniform probability
measure on $[0,1]$. Indeed, reduction modulo one is injective on
$(0,1)$, while the two endpoints have the same image and that point
has Haar measure zero.

Choose $a\geq1$ with $aG\subseteq\Gamma$. By
Lemma~\ref{lem:periodic-correction},
$g=P_{a\Lambda_i}f-f_i$ has a finite-index period subgroup $\Pi$.
Let $E_1,\ldots,E_m$ be the ergodic components of
$\Delta=\Gamma\cap\Pi$ contained in $E$. On each $E_t$, the
$\Delta$-invariant function $g$ is a constant $c_t$, and
\eqref{eq:hereditary} applies because $\Delta\leq\Gamma_i$. If
$\nu_t=(\overline f_i)_*\mu_{E_t}$, then
\[
 (\overline f_i)_*\mu_E
 =\sum_{t=1}^m\frac{\mu(E_t)}{\mu(E)}\nu_t.
\]
Every coefficient is positive. A Dirac constituent would give the
left side an atom, so all the $\nu_t$ are Haar. Consequently $f_i$
is uniform on $[0,1]$ on every $E_t$. Since
\[
 0\leq f_i+c_t=P_{a\Lambda_i}f\leq1
\]
almost everywhere by positivity and unitality of conditional
expectation, the essential endpoints give $c_t\geq0$ and
$1+c_t\leq1$, hence $c_t=0$. Thus
$P_{a\Lambda_i}f=f_i$ on $E$. Finally,
$a\Lambda_i\leq\Gamma$ makes $1_E$ measurable with respect to
$\mathcal I_{a\Lambda_i}$, and conditional expectation gives
\[
 \mu_E(D)=\frac{1}{\mu(E)}\int1_Ef\,d\mu
 =\frac{1}{\mu(E)}\int1_EP_{a\Lambda_i}f\,d\mu
 =\int_Ef_i\,d\mu_E=\frac12.
\]

\smallskip\noindent\emph{Case 2: all conditional laws are Dirac.}
Suppose first that $(\overline f_i)_*\mu_E=\delta_0$ for some $i$.
Then $E\subseteq A_0\cup A_1$ modulo null sets, where
$A_0=\{f_i=0\}$ and $A_1=\{f_i=1\}$. These sets belong to
$\mathcal I_{\Lambda_i}$, so
\[
 \int_{A_0}f\,d\mu=\int_{A_0}f_i\,d\mu=0,
 \qquad
 \int_{A_1}(1-f)\,d\mu
 =\int_{A_1}(1-f_i)\,d\mu=0.
\]
The integrands are nonnegative. Hence $f=f_i$ on $A_0\cup A_1$, and
$1_{D\cap E}=1_Ef_i$ is invariant under
$\Gamma\cap\Lambda_i$. Since $\Gamma$ has finite index in $G$, the
quotient $\Lambda_i/(\Gamma\cap\Lambda_i)$ injects into $G/\Gamma$.
Thus $\Gamma\cap\Lambda_i$ has finite index in $\Lambda_i$, hence
rank $d-1$, and $D\cap E$ is $(d-1)$-periodic.

Otherwise every Dirac mass is at a nonzero point of
$\mathbb R/\mathbb Z$. Its unique representative in $(0,1)$ is the
almost-sure value of $f_i$ on $E$. Also $f_0$ is constant on $E$,
since $\Gamma\leq\Gamma_0$. Thus
$f=f_0+\sum_{i=1}^sf_i$ is constant on $E$. As $f$ is binary,
$D\cap E$ is either empty or all of $E$ modulo null sets. Hence
$D\cap E$ is $\Gamma$-invariant modulo null sets. Since $\Gamma$ has
rank $d$, it contains a rank-$(d-1)$ subgroup, so $D\cap E$ is
$(d-1)$-periodic.
\end{proof}

\subsection{Periodic replacement of half-density components}

The conditional-density dichotomy leaves components on which the
coordinate event has density $1/2$. We first establish closure under
relative complements, then use the tiling equation to pair the
remaining translates. Compare~\cite[Lemmas~C.16--C.17]{Khetan}.

\begin{lemma}[Relative complements]\label{lem:complement}
Let $d\geq1$, set $G=\mathbb Z^d$, and let $G$ act ergodically by
probability-preserving transformations on $(X,\mathcal B,\mu)$. If a
measurable $d$-periodic set $E$ contains a piecewise
$(d-1)$-periodic set $B$, then $E\setminus B$ is piecewise
$(d-1)$-periodic.
\end{lemma}
\begin{proof}
Write $B=\bigsqcup_{i=1}^mB_i$, with rank-$(d-1)$ period groups
$L_i$. All set identities in this proof are understood modulo null
sets.

\emph{Step 1: separate the period directions.}
Group the $B_i$ according to the rational spans
$L_i\otimes\mathbb Q$. Within each group, replace the pieces by their
union and their period groups by their intersection. Subgroups with
the same rational span are commensurable, so this intersection still
has rank $d-1$ and preserves the union. We may therefore assume that
the spaces $L_i\otimes\mathbb Q$ are pairwise distinct. Two distinct
hyperplanes span $\mathbb Q^d$, so $L_i+L_j$ has full rank and hence
finite index for $i\ne j$.

Choose a rank-$d$ period subgroup $P$ of $E$, and put
\[
 \Delta=P\cap\bigcap_{1\leq i<j\leq m}(L_i+L_j),
\]
with an empty intersection interpreted as $G$. This subgroup has
finite index. By Lemma~\ref{lem:components}, $E$ is a union of
finitely many $\Delta$-ergodic components.

\smallskip\noindent\emph{Step 2: only one period direction survives on each component.}
Let $Q\subseteq E$ be such a component, and suppose
$\mu(Q\cap B_j)>0$. For $i\ne j$, consider
\[
 \mathcal S=\bigcup_{h\in L_i+L_j}T_h(Q\cap B_j).
\]
Since $\Delta\subseteq L_i+L_j$, the set $\mathcal S$ is
$\Delta$-invariant. Its intersection with $Q$ has positive measure,
so ergodicity on $Q$ gives $Q\subseteq\mathcal S$. For
$h=h_i+h_j$, with $h_i\in L_i$ and $h_j\in L_j$,
\[
 B_i\cap T_h(Q\cap B_j)
 \subseteq B_i\cap T_{h_i}B_j
 =T_{h_i}(B_i\cap B_j)=\varnothing.
\]
The defining union is countable, hence $B_i\cap\mathcal S$ is null
and $\mu(Q\cap B_i)=0$. Thus at most one $B_i$ meets $Q$ in positive
measure.

If no $B_i$ does, then $Q\setminus B=Q$ is $\Delta$-invariant.
Otherwise $Q\setminus B=Q\setminus B_j$ is invariant under
$\Delta\cap L_j$. The quotient $L_j/(\Delta\cap L_j)$ injects into
$G/\Delta$, so this intersection has finite index in $L_j$ and
therefore rank $d-1$. In the first case, the full-rank group $\Delta$
contains a rank-$(d-1)$ subgroup. Taking relative complements on the
finitely many disjoint components $Q$ proves the assertion.
\end{proof}

The next proposition carries out the replacement on a single orbit.
The finite component partition determines a periodic graph, and
Lemma~\ref{lem:color} makes all its paired choices simultaneously.

\begin{proposition}[Replacement of half-density components]\label{prop:replacement}
Let $d\geq1$, set $G=\mathbb Z^d$, and let $F\subseteq G$ be a
nonempty finite set whose tiling space $X_F$ is nonempty. Let $\mu$
be a $G$-invariant Borel probability measure on $X_F$ for which the
shift action is ergodic, and let
\[
 D=\{a\in X_F:a(0)=1\}.
\]
Suppose there is a finite-index subgroup $\Gamma\leq G$ such that,
for every $\Gamma$-ergodic component $E$, either $D\cap E$ is
piecewise $(d-1)$-periodic or $\mu_E(D)=1/2$ (or both). Then there is
a piecewise $(d-1)$-periodic set $A'\subseteq G$ such that
$F\oplus A'=G$.
\end{proposition}
\begin{proof}
Let $\mathcal E$ be the finite set of $\Gamma$-ergodic components
and write $D_E=D\cap E$. Call $E$ \emph{ordinary} if $D_E$ is
piecewise $(d-1)$-periodic, and \emph{exceptional} otherwise. Every
exceptional component has density $1/2$; an ordinary component may
have that density as well.

\emph{Step 1: each component has zero or two exceptional translates.}
For $E\in\mathcal E$, define
$I(E)=\{u\in F:T_{-u}E\text{ is exceptional}\}$. Restricting the
tiling partition from Lemma~\ref{lem:coordinate-identities} to $E$
gives
\begin{equation}\label{eq:partition}
 E=\bigsqcup_{u\in F}T_uD_{T_{-u}E}.
\end{equation}
Indeed, $E\cap T_uD=T_u(D\cap T_{-u}E)$. Translation preserves
$\mu$ and takes $T_{-u}E$ onto $E$, so
\[
 \mu_E(T_uD_{T_{-u}E})
 =\frac{\mu(D\cap T_{-u}E)}{\mu(E)}
 =\mu_{T_{-u}E}(D).
\]
Every term indexed by $I(E)$ therefore has mass $1/2$. Summing these
nonnegative masses in \eqref{eq:partition} gives $|I(E)|\leq2$.

If $I(E)=\{u\}$, let $B$ be the union of the other terms in
\eqref{eq:partition}. Each is a translate of a piecewise
$(d-1)$-periodic set, and they are disjoint, so $B$ has the same
property. The set $E$ is $\Gamma$-invariant, and $\Gamma$ has rank
$d$, so $E$ is $d$-periodic. Lemma~\ref{lem:complement} now implies
that $E\setminus B=T_uD_{T_{-u}E}$ is piecewise
$(d-1)$-periodic. Translating by $-u$ shows that
$D_{T_{-u}E}$ is piecewise $(d-1)$-periodic, contrary to the
definition of the exceptional component $T_{-u}E$. Hence
$|I(E)|\in\{0,2\}$. When $|I(E)|=2$, the two exceptional translates
already have total mass one, so all ordinary terms in
\eqref{eq:partition} are null.

\smallskip\noindent\emph{Step 2: make the identities pointwise on an orbit.}
For each ordinary component, choose a finite disjoint decomposition
$D_E=\bigsqcup_jB_{E,j}$ with rank-$(d-1)$ period groups $L_{E,j}$.
Replace each $L_{E,j}$ by $L_{E,j}\cap\Gamma$. Since $\Gamma$ has
finite index in $G$, the quotient
$L_{E,j}/(L_{E,j}\cap\Gamma)$ injects into $G/\Gamma$. Hence this
intersection has finite index in $L_{E,j}$ and therefore the same
rank.

Choose measurable representatives of these pieces and of the
components. Remove the null sets where their partition identities,
period identities, or component permutation identities fail, as well
as the ordinary terms declared null in Step~1. Remove every
$G$-translate of these null sets as well. There are finitely many
pieces and countably many group elements, so the remaining set $X_*$
is invariant and conull. Every required identity now holds at every
point of $X_*$ and along its entire orbit.

Fix $\omega\in X_*$ and set
\[
 a(n)=1_D(T_n\omega)=\omega(n),
 \qquad
 E_n=\text{the component containing }T_n\omega.
\]
For all $n,g\in G$, the component identities give
$E_{n+g}=T_gE_n$. In particular,
$E_{n+\gamma}=E_n$ for $\gamma\in\Gamma$, and
$E_{n-u}=T_{-u}E_n$ for $u\in F$. The tiling equation is
$\sum_{u\in F}a(n-u)=1$. For every ordinary piece, define
\[
 A_{E,j}=\{n\in G:T_n\omega\in B_{E,j}\}.
\]
Then $A_{E,j}+L_{E,j}=A_{E,j}$. These sets are pairwise disjoint,
and their union is exactly
\[
 \{n:E_n\text{ is ordinary and }a(n)=1\}.
\]

\smallskip\noindent\emph{Step 3: construct and color the periodic graph.}
Let $V=\{n\in G:E_n\text{ is exceptional}\}$. Whenever
$I(E_n)=\{u,v\}$, put the undirected edge $\{n-u,n-v\}$ in a graph
on $V$. Both endpoints lie in $V$, and they are distinct because
$u\ne v$. The $\Gamma$-periodicity of $n\mapsto E_n$ implies that
translating an edge by $\gamma\in\Gamma$ gives the edge associated
with $n+\gamma$. Thus the vertex and edge sets are $\Gamma$-periodic.

For such $n$, Step~1 and the choice of $X_*$ give $a(n-w)=0$ for
every $w\in F\setminus\{u,v\}$. Therefore
$a(n-u)+a(n-v)=1$, and $a|_V$ is a proper two-coloring.
Lemma~\ref{lem:color} supplies a proper coloring
$c\colon V\to\{0,1\}$ invariant under a finite-index subgroup
$K\leq\Gamma$.

\smallskip\noindent\emph{Step 4: verify the replacement and its periods.}
Define
\[
 a'(n)=
 \begin{cases}
  c(n),&n\in V,\\
  a(n),&n\notin V.
 \end{cases}
\]
If $I(E_n)$ is empty, then $n-w\notin V$ for every $w\in F$, and
\[
 \sum_{w\in F}a'(n-w)=\sum_{w\in F}a(n-w)=1.
\]
If $I(E_n)=\{u,v\}$, then, for every
$w\in F\setminus\{u,v\}$,
\[
 a'(n-w)=a(n-w)=0.
\]
Proper coloring therefore gives
\[
 \sum_{w\in F}a'(n-w)=c(n-u)+c(n-v)=1.
\]
Thus $1_F*a'=1$ at every $n\in G$.

The set $C=\{n\in V:c(n)=1\}$ is $K$-periodic, since $V$ is
$\Gamma$-periodic and $K\leq\Gamma$. Set
\[
 A'=\{n\in G:a'(n)=1\}.
\]
Then
\[
 A'=C\sqcup\bigsqcup_{E,j}A_{E,j}.
\]
Every $A_{E,j}$ has a rank-$(d-1)$ period group. The finite-index
group $K$ has rank $d$ and contains a rank-$(d-1)$ subgroup, so $C$
has the required periodicity as well. Hence $A'$ is piecewise
$(d-1)$-periodic, while $1_F*a'=1$ says exactly that
$F\oplus A'=G$.
\end{proof}

\begin{proof}[Proof of Theorem~\ref{thm:circle}]
Theorem~\ref{thm:dichotomy} gives a finite-index component partition
satisfying the hypotheses of Proposition~\ref{prop:replacement}.
That proposition produces a complement for $F$ which is a finite
disjoint union of sets, each invariant under a rank-$(d-1)$ subgroup
of $\mathbb Z^d$. The single-tile case of the periodicization theorem of Meyerovitch,
Sanadhya, and Solomon~\cite[Theorem~1.3]{MSS} states that a finite tile
admitting such a complement also admits a fully periodic complement.
Applying it to $F$ proves the theorem.
\end{proof}

\section{Nine-point tiles}\label{sec:nine}

For nine-point tiles, the dilation equations have a particularly
simple block structure. We first identify the case that gives a
lattice complement directly, then apply the spectral-circle criterion
to the remaining case.

\begin{proposition}\label{prop:nine}
For every integer $d\geq1$, every set $F\subseteq\mathbb Z^d$ with
$|F|=9$ that tiles $\mathbb Z^d$ by translations admits a fully
periodic tiling complement.
\end{proposition}
\begin{proof}
\emph{Step 1: the sizes of the zero-sum blocks.}
Fix an integer $d\geq1$ and a set $F\subseteq\mathbb Z^d$ with
$|F|=9$ that tiles $\mathbb Z^d$ by translations. By
Lemma~\ref{lem:intrinsic-reduction}, after translating $F$ we may
assume that $0\in F$, put
$G=\langle F\rangle\cong\mathbb Z^r$, and restrict the given tiling
to $G$. For $r=1$, Newman's theorem~\cite{Newman} gives a fully
periodic complement in $G$. For $r=2$, Bhattacharya
\cite[Theorem~1.1]{Bhattacharya} gives one. In either case,
Lemma~\ref{lem:extend} gives a fully periodic complement in
$\mathbb Z^d$. We may therefore assume $r\geq3$.

Take $Q=3$ in~\eqref{eq:general-zeros}; all the dilation factors
$1+3j$ are coprime to $9$. For each $\theta\in Z_{F,3}$, let
$\mathcal B_\theta$ be the partition of $F$ by equality of
$\theta(u)^3$, let $b_\theta=|\mathcal B_\theta|$, and let
$H_\theta$ be the subgroup generated by the differences of points
within its blocks. Lemma~\ref{lem:blocks} shows that all these blocks
have zero phase sum.

For any $\theta\in Z_{F,3}$, any block $B\in\mathcal B_\theta$,
and any base point $a\in B$, each phase $\theta(u-a)$ belongs to
$\{1,\zeta,\zeta^2\}$, where $\zeta=\exp(2\pi i/3)$. Let $n_j$
count the occurrences of $\zeta^j$. Dividing the zero-sum equation
for $B$ by $\theta(a)$ gives
\[
 0=n_0+n_1\zeta+n_2\zeta^2
   =(n_0-n_2)+(n_1-n_2)\zeta.
\]
Since $1$ and $\zeta$ are linearly independent over $\mathbb Q$,
we have $n_0=n_1=n_2$. Thus every block has positive size divisible
by three, so $b_\theta\leq3$. Lemma~\ref{lem:blocks} also gives, for
every $\theta\in Z_{F,3}$,
\[
 \operatorname{rank}H_\theta\geq r-b_\theta+1\geq r-2,
 \qquad \theta\in(3H_\theta)^\perp.
\]

\smallskip\noindent\emph{Step 2: a lattice complement in the rank-$r-2$ case.}
Suppose there is a $\theta\in Z_{F,3}$ with
$\operatorname{rank}H_\theta=r-2$. Fix such a $\theta$ and abbreviate
$\mathcal B=\mathcal B_\theta$, $b=b_\theta$, and $H=H_\theta$.
The rank inequality forces $b=3$. Since the three block sizes are
positive multiples of three with sum nine, every block has size three.

Write the blocks as $B_0,B_1,B_2$, with bases
$a_0=0,a_1,a_2$. The character $\theta$ takes values in
$\{1,\zeta,\zeta^2\}$ on every generator of $H$, hence on
all of $H$. The equation
$\theta(h)=\zeta^{\chi(h)}$ therefore defines a homomorphism
$\chi\colon H\to\mathbb Z/3\mathbb Z$. Step~1 shows that each
normalized block $B_j-a_j$ has one point of each phase, so
$\chi$ maps $B_j-a_j$ bijectively onto $\mathbb Z/3\mathbb Z$.

To make the resulting quotient explicit, the rank argument in
Lemma~\ref{lem:quotient} gives $G=H\oplus\mathbb Za_1
\oplus\mathbb Za_2$. Define
\[
 \Psi(h+m_1a_1+m_2a_2)
   =\bigl(\chi(h),\,m_1+2m_2\bmod3\bigr)
   \in(\mathbb Z/3\mathbb Z)^2.
\]
For $j=0,1,2$, the restriction of $\Psi$ to $B_j$ is a
bijection onto $(\mathbb Z/3\mathbb Z)\times\{j\}$.
Consequently $\Psi|_F$ is bijective, and
$F\oplus\ker\Psi=G$ by Lemma~\ref{lem:quotient}.
Thus $\ker\Psi$ is a lattice complement, and
Lemma~\ref{lem:lattice-index} gives
$[G:\ker\Psi]=|F|=9$. Since $\ker\Psi$ is a subgroup, it is
invariant under itself and hence is a fully periodic complement.

\smallskip\noindent\emph{Step 3: spectral circles in the remaining case.}
Otherwise, $\operatorname{rank}H_\theta\geq r-1$ for every
$\theta\in Z_{F,3}$. Let
\[
 \mathcal P=\{\mathcal B_\theta:\theta\in Z_{F,3}\}.
\]
This is a finite collection of partitions of $F$. For each
$\mathcal B\in\mathcal P$, let $H_{\mathcal B}$ be the subgroup
generated by differences within its blocks. If $Z_{F,3}=\varnothing$,
then $\mathcal P=\varnothing$ and the union below is empty. Otherwise,
Lemma~\ref{lem:blocks} yields
\[
 Z_{F,3}\subseteq
    \bigcup_{\mathcal B\in\mathcal P}(3H_{\mathcal B})^\perp,
 \qquad
 \operatorname{rank}H_{\mathcal B}\geq r-1
 \quad(\mathcal B\in\mathcal P).
\]
For each $\mathcal B\in\mathcal P$, apply
\eqref{eq:annihilator-coordinates} to $L=3H_{\mathcal B}$. Its rank
is $r-1$ or $r$, so its annihilator is, respectively, a finite union
of rational affine circles or a finite torsion set.

Choose an isomorphism $G\cong\mathbb Z^r$ and henceforth use it to
identify the two groups and their duals. Since $r\geq3$, every torsion
point in the finite sets above can be included in a rational affine
circle. Choose an ergodic invariant measure on the nonempty tiling
space $X_F$, as in Subsection~\ref{subsec:systems}. By
Lemma~\ref{lem:common-support}, the spectral measure of its coordinate
indicator is supported on $Z_{F,3}\cup\{\mathbf1\}$. The preceding
description, with an additional rational affine circle through
$\mathbf1$ if needed, verifies the hypothesis of
Theorem~\ref{thm:circle}. Hence $F$ has a fully periodic complement
in $G$.

In Steps~2 and~3, Lemma~\ref{lem:extend} extends the complement
to $\mathbb Z^d$. Undoing the initial translation completes the proof.
\end{proof}

\begin{proof}[Completion of the proof of Theorem~\ref{thm:main}]
Subsection~\ref{subsec:twice-completion} proves the result for $|F|=2q$.
Proposition~\ref{prop:nine} proves it for $|F|=9$.
\end{proof}

\section{Periods and decidability}\label{sec:effective}

The existence of a periodic complement has two further consequences.
We can remove prime factors of its scalar period that do not divide
the tile cardinality, and we can search the resulting finite
quotients to decide tileability.

\subsection{Prime-supported periods}

The one-dimensional period reduction is due to Coven and
Meyerowitz~\cite[Lemma~2.3]{CM}; see also \L aba and
Zakharov~\cite[Lemma~1]{LZ}. The finite-quotient argument below
works in arbitrary rank.

\begin{proposition}[Prime-supported periods]\label{prop:prime-periods}
Every finite set $F\subseteq G\cong\mathbb Z^r$ of cardinality $n$ that
has a fully periodic complement also has an $sG$-periodic complement
for some integer $s$ all of whose prime divisors divide $n$.
Equivalently, it has an $n^kG$-periodic complement for some $k\geq0$.
\end{proposition}

\begin{proof}
\emph{Step 1: dilation in a finite quotient.}
Let $P\leq G$ be a finite-index period subgroup for a complement.
Choose $M\geq1$ with $MG\subseteq P$, and write $M=st$, where
$s$ contains precisely the full prime-power factors of $M$
whose primes divide $n$. Then $\gcd(t,n)=1$.

Put $Q=G/MG$, and denote the image of $u\in G$ in $Q$ by
$\bar u$. The complement descends to a function
$\bar a\colon Q\to\{0,1\}$ satisfying
$\sum_{u\in F}\bar a(y-\bar u)=1$ for every $y\in Q$.
Here and below the sum is indexed by the original points of $F$,
so any coincidences in a quotient retain their multiplicities.
Lemma~\ref{lem:dilation}, applied to this finite-group tiling,
gives
\[
 \sum_{u\in F}\bar a(y-t\bar u)=1\qquad(y\in Q).
\]
Every $t\bar u$ belongs to $tQ$. Restricting the equation to
$y\in tQ$ therefore uses only values of
$b=\bar a|_{tQ}$ and gives a tiling equation on $tQ$.

\smallskip\noindent\emph{Step 2: removal of the factor $t$.}
Consider the map
\[
 \phi\colon G/sG\longrightarrow tQ,
 \qquad \phi(x+sG)=tx+MG.
\]
It is well-defined because $t(sG)=MG$, and is surjective by
the definition of $tQ$. If $\phi(x+sG)=0$, then
$tx=stz$ for some $z\in G$. Since $G$ is torsion-free,
$x=sz$, proving injectivity.

Pull $b$ back under this isomorphism: set
$c(x+sG)=b(tx+MG)$. For every $x\in G$, the restricted
dilation equation becomes
\[
 \sum_{u\in F}c(x-u+sG)
   =\sum_{u\in F}b(tx-tu+MG)=1.
\]
Thus $a'(x)=c(x+sG)$ is a binary, $sG$-periodic tiling
indicator on $G$.

This equation also rules out collisions among tile points modulo
$sG$. Indeed, $c$ takes the value one somewhere, say at $z+sG$.
If distinct $u,v\in F$ had $u-v\in sG$, the equation at
$x=z+u$ would contain the two summands
$c(z+sG)=c(z+u-v+sG)=1$, a contradiction.
Finally, every prime divisor of $s$ divides $n$, so $s\mid n^k$
for some $k$. Since $n^kG\subseteq sG$, the same complement
is $n^kG$-periodic.
\end{proof}

\begin{corollary}\label{cor:prime-periods}
Let $d\in\mathbb Z_{\geq1}$, and let $F\subseteq\mathbb Z^d$ be a
finite set that tiles $\mathbb Z^d$ by translations. If $|F|=2q$
for an odd prime $q$, then $F$ has a
$(2q)^k\mathbb Z^d$-periodic complement for some
$k\in\mathbb Z_{\geq0}$. If $|F|=9$, then $F$ has a
$3^k\mathbb Z^d$-periodic complement for some
$k\in\mathbb Z_{\geq0}$.
\end{corollary}
\begin{proof}
By Theorem~\ref{thm:main}, $F$ has a fully periodic tiling complement.
Applying Proposition~\ref{prop:prime-periods} with
$G=\mathbb Z^d$ and $n=|F|$ gives an
$|F|^m\mathbb Z^d$-periodic complement for some
$m\in\mathbb Z_{\geq0}$. If $|F|=2q$, take $k=m$.
If $|F|=9$, then $|F|^m=9^m=3^{2m}$, so take
$k=2m\in\mathbb Z_{\geq0}$.
\end{proof}

\subsection{Decidability}

For the cardinalities in Theorem~\ref{thm:main}, tileability has
a finite certificate given by a periodic complement. Nontileability
has a finite certificate given by an inconsistent collection of
covering equations. Searching for both gives the following procedure.

\begin{corollary}\label{cor:decidable}
There is an algorithm which, given $d\geq1$ and a finite set
$F\subseteq\mathbb Z^d$ with $|F|=2q$ for an odd prime $q$
or $|F|=9$, decides whether $F$ tiles $\mathbb Z^d$.
If $F$ tiles, the algorithm produces a fully periodic complement.
\end{corollary}

\begin{proof}
\emph{Step 1: finite certificates for tileability.}
For $k=0,1,\ldots$, put $m_k=(2q)^k$ in the twice-prime case
and $m_k=3^k$ in the nine-point case. On
$Q_k=(\mathbb Z/m_k\mathbb Z)^d$, test all binary functions
$a\colon Q_k\to\{0,1\}$ for the finite system
\[
 \sum_{u\in F}a(x-u)=1\qquad(x\in Q_k),
\]
where subtraction is taken modulo $m_k$ and all summands are
retained. Each test is finite. A solution lifts to the periodic
complement $\{x\in\mathbb Z^d:a(x\bmod m_k)=1\}$.
If $F$ tiles, Corollary~\ref{cor:prime-periods} guarantees
that a solution occurs for some $k$.

\smallskip\noindent\emph{Step 2: finite certificates for nontileability.}
For $R=0,1,\ldots$, let $B_R=[-R,R]^d\cap\mathbb Z^d$.
Introduce one binary variable $a_y$ for each $y\in B_R-F$,
and test the finite system
\[
 \sum_{u\in F}a_{x-u}=1\qquad(x\in B_R).
\]
If it is inconsistent, no global tiling exists. Conversely,
suppose it is consistent for every $R$. In the compact product
space $\Omega=\{0,1\}^{\mathbb Z^d}$, define
\[
 C_R=\left\{a\in\Omega:
          \sum_{u\in F}a(x-u)=1\text{ for every }x\in B_R\right\}.
\]
Each $C_R$ is closed because its defining conditions involve
only finitely many coordinates. It is nonempty because a
solution on $B_R-F$ can be extended arbitrarily to the other
coordinates. Also $C_{R+1}\subseteq C_R$. Compactness gives
$\bigcap_{R\geq0}C_R\ne\varnothing$; every configuration in
this intersection satisfies the tiling equation at every point
of $\mathbb Z^d$. Therefore, if $F$ does not tile, the search
must reach an inconsistent finite system.

\smallskip\noindent\emph{Step 3: termination.}
At stage $j$, perform the finite test of Step~1 with $k=j$ and
the finite test of Step~2 with $R=j$, unless an earlier test has
already terminated the algorithm. A successful test in Step~1
returns a periodic complement; an inconsistent system in Step~2
returns nontileability. These conclusions cannot conflict, and
Steps~1 and~2 show that one is eventually reached.
\end{proof}

The compactness argument is standard; see~\cite{Bhattacharya,Szegedy}.
An explicit bound for the successful period exponent would give a
corresponding quantitative complexity bound.

\section*{Declaration of generative AI use}
OpenAI's ChatGPT (Codex) was used in the development and preparation of this manuscript for mathematical exploration, proof drafting and checking, literature searches, and language editing. Responsibility for the results and their presentation rests with the authors.


\begin{thebibliography}{99}
\small
\setlength{\itemsep}{3pt plus 1pt minus 1pt}

\bibitem{PeriodicGraphs}
T. Abrishami, L. Esperet, U. Giocanti, M. Hamann, P. Knappe, and
R.~G. M\"oller,
\emph{Periodic colorings and orientations in infinite graphs},
Combinatorial Theory \textbf{5} (2025), no.~4, Paper No.~5.

\bibitem{Bhattacharya}
S. Bhattacharya,
\emph{Periodicity and decidability of tilings of $\mathbb Z^2$},
American Journal of Mathematics \textbf{142} (2020), no.~1, 255--266.

\bibitem{Chen}

R. Chen,
\emph{Decompositions and measures on countable Borel equivalence relations},
Ergodic Theory and Dynamical Systems \textbf{41} (2021), no.~12,
3671--3703.


\bibitem{CM}
E.~M. Coven and A. Meyerowitz,
\emph{Tiling the integers with translates of one finite set},
Journal of Algebra \textbf{212} (1999), no.~1, 161--174.

\bibitem{DasNg}

S. Das and J. Ng,
\emph{Lecture 21: Smith Normal Form},
Algebra II Student Notes, MIT OpenCourseWare RES.18-012, Spring 2022,
\href{https://ocw.mit.edu/courses/res-18-012-algebra-ii-student-notes-spring-2022/mit18_702s22_lect21.pdf}
{lecture notes (PDF)}.


\bibitem{deBruijn}
N.~G. de Bruijn,
\emph{On the factorization of cyclic groups},
Nederl. Akad. Wetensch. Proc. Ser.~A \textbf{56}
(= Indagationes Mathematicae \textbf{15}) (1953), 370--377.

\bibitem{GGR}
J. Greb\'{\i}k, R. Greenfeld, V. Rozho\v{n}, and T. Tao,
\emph{Measurable tilings by abelian group actions},
International Mathematics Research Notices \textbf{2023} (2023),
no.~23, 20211--20251.

\bibitem{GreenfeldSurvey}
R. Greenfeld,
\emph{Translational tilings: structured or wild?},
in \emph{Proceedings of the International Congress of Mathematicians
2026, Vol.~3: Invited Lectures (Sections 1--4)},
S. Friedlander and Y. Tschinkel (eds.),
SIAM, Philadelphia, 2026, 76--96.

\bibitem{GTCounterexample}
R. Greenfeld and T. Tao,
\emph{A counterexample to the periodic tiling conjecture},
Annals of Mathematics (2) \textbf{200} (2024), no.~1, 301--363.

\bibitem{GTStructure}
R. Greenfeld and T. Tao,
\emph{The structure of translational tilings in $\mathbb Z^d$},
Discrete Analysis (2021), Paper No.~16, 28 pp.

\bibitem{GrunbaumShephard}
B. Gr\"unbaum and G.~C. Shephard,
\emph{Tilings and Patterns},
W.~H. Freeman and Company, New York, 1987.

\bibitem{Hatcher}

A. Hatcher,
\emph{Algebraic Topology},
Cambridge University Press, Cambridge, 2002,
\href{https://pi.math.cornell.edu/~hatcher/AT/AT.pdf}{author's PDF}.


\bibitem{HorakKim}
P. Horak and D. Kim,
\emph{Algebraic method in tilings},
preprint, 2016,
\href{https://arxiv.org/abs/1603.00051}{arXiv:1603.00051}.

\bibitem{Khetan}
A. Khetan,
\emph{A periodicity result for tilings of $\mathbb Z^3$ by clusters
of prime-squared cardinality},
preprint, 2021, revised 2026,
\href{https://arxiv.org/abs/2109.14179v3}{arXiv:2109.14179v3}.

\bibitem{LZ}
I. \L aba and D. Zakharov,
\emph{On the minimal period of integer tilings},
Bulletin of the London Mathematical Society \textbf{57} (2025),
no.~4, 1160--1170.

\bibitem{LagariasWang}
J.~C. Lagarias and Y. Wang,
\emph{Tiling the line with translates of one tile},
Inventiones Mathematicae \textbf{124} (1996), no.~1--3, 341--365.

\bibitem{LamLeung}
T.~Y. Lam and K.~H. Leung,
\emph{On vanishing sums of roots of unity},
Journal of Algebra \textbf{224} (2000), no.~1, 91--109.

\bibitem{Lindenstrauss}
E. Lindenstrauss,
\emph{Pointwise theorems for amenable groups},
Inventiones Mathematicae \textbf{146} (2001), no.~2, 259--295.

\bibitem{MSS}
T. Meyerovitch, S. Sanadhya, and Y. Solomon,
\emph{Periodicity of joint co-tiles in $\mathbb Z^d$},
Discrete Analysis (2024), Paper No.~13, 32 pp.

\bibitem{Newman}
D.~J. Newman,
\emph{Tesselation of integers},
Journal of Number Theory \textbf{9} (1977), no.~1, 107--111.

\bibitem{PoonenRubinstein}
B. Poonen and M. Rubinstein,
\emph{The number of intersection points made by the diagonals of a
regular polygon},
SIAM Journal on Discrete Mathematics \textbf{11} (1998), no.~1,
135--156.

\bibitem{Rudin}
W. Rudin,
\emph{Fourier Analysis on Groups},
Interscience Tracts in Pure and Applied Mathematics, vol.~12,
Interscience Publishers, New York, 1962.

\bibitem{Stein}
S.~K. Stein,
\emph{Algebraic tiling},
American Mathematical Monthly \textbf{81} (1974), no.~5, 445--462.

\bibitem{Szegedy}
M. Szegedy,
\emph{Algorithms to tile the infinite grid with finite clusters},
in \emph{Proceedings of the 39th Annual IEEE Symposium on
Foundations of Computer Science (FOCS 1998)},
IEEE Computer Society, 1998, 137--145.

\bibitem{Weyl}
H. Weyl,
\emph{\"Uber die Gleichverteilung von Zahlen mod. Eins},
Mathematische Annalen \textbf{77} (1916), no.~3, 313--352.

\end{thebibliography}
\end{document}